\documentclass[a4paper,12pt]{amsart}
\usepackage{amsmath,amsfonts,amssymb,amsthm}

\usepackage{subfig}
\usepackage{graphicx}
\usepackage{color}
\usepackage[all]{xy}
\usepackage{url}
\usepackage{setspace}
\usepackage{booktabs}
\usepackage{longtable}
\usepackage{latexsym}
\usepackage{comment}
\usepackage[usenames,dvipsnames]{xcolor}
\usepackage{hyperref}
\usepackage{mathtools}

\usepackage{tikz}
\usepackage{multido}
\usepackage{tabu}
\usetikzlibrary{cd}
\usetikzlibrary{patterns}
\usetikzlibrary{math}
\usetikzlibrary{shapes}
\usetikzlibrary{decorations.fractals}

\newtheorem{theorem}{Theorem}[section]
\newtheorem{lemma}[theorem]{Lemma}
\newtheorem{proposition}[theorem]{Proposition}
\newtheorem{corollary}[theorem]{Corollary}

\theoremstyle{definition}
\newtheorem{definition}[theorem]{Definition}
\newtheorem{example}[theorem]{Example}

\newtheorem{remark}[theorem]{Remark}

\definecolor{cKlaus}{rgb}{0.1,0.55,0.03}
\definecolor{cKlausOK}{rgb}{0.6,0.10,0.33}
\definecolor{intOrange}{rgb}{1.0,.310,.0}

\newcommand{\ktrash}[1]{}

\usepackage[textsize=tiny]{todonotes}

\definecolor{cMartin}{rgb}{0.2,0.65,0.33}

\newcommand{\kbox}[1]{#1}

\newcommand{\cN}{\mathcal{N}}

\DeclareMathOperator{\rank}{rk}

\newcommand{\PP}{\mathbb P}

\newcommand{\N}{\mathbb N}
\newcommand{\R}{\mathbb R}

\newcommand{\T}{\mathbb T}
\newcommand{\V}{\mathbb V}
\newcommand{\Z}{\mathbb Z}

\newcommand{\CD}{{\mathcal D}}

\newcommand{\CL}{{\mathcal L}}
\newcommand{\CO}{{\mathcal O}}

\newcommand{\toric}{\T\V}  
\newcommand{\gH}{\operatorname{H}}

\newcommand{\Pic}{\operatorname{Pic}}

\newcommand{\Cl}{\operatorname{Cl}}
\definecolor{intOrange}{rgb}{1.0,.310,.0}

\newcommand{\gHom}{\mbox{\rm Hom}}

\newcommand{\disjcup}{\sqcup}

\newcommand{\kst}{\,|\;}

\newcommand{\surj}{\rightarrow\hspace{-0.8em}\rightarrow}

\newcommand{\ko}{\overline}

\newcommand{\normal}{\cN}

\newcommand{\vect}{\overrightarrow}

\newcommand{\Imm}{\operatorname{Imm}}
\newcommand{\cl}{{s}}

\newcommand{\mes}{{maximal exceptional sequence}}
\newcommand{\fes}{{full exceptional sequence}}

\newcommand{\cont}{{\operatorname{cont}}}
\newcommand{\width}{{w}}

\newif\ifcomment

\newcommand{\scc}{\sc} 
\definecolor{originM}{RGB}{180,0,0}
\definecolor{origin}{RGB}{0,130,0}
\definecolor{regi}{HTML}{677081}
\definecolor{regii}{RGB}{0,0,160}
\definecolor{regiii}{RGB}{0,0,160}
\definecolor{redSeg}{RGB}{160,0,0}
\definecolor{center}{RGB}{180,120,60}
\definecolor{parallelogram}{RGB}{135,206,250}
\definecolor{coverA}{RGB}{180,0,0}
\definecolor{coverB}{RGB}{10,180,0}
\definecolor{coverC}{RGB}{10,0,180}
\definecolor{coverD}{RGB}{80,20,60}
\definecolor{coverE}{RGB}{20,60,10}
\definecolor{coverF}{RGB}{120,50,170}
\definecolor{coverG}{RGB}{20,50,240}
\definecolor{coverH}{rgb}{1.0,0.25,0.0}

\definecolor{skin}{HTML}{FFECC9}
\definecolor{pumpkin}{HTML}{FEDFA9}
\definecolor{piggy}{HTML}{FFB99D}
\definecolor{fiolet}{HTML}{CD8F9C}
\definecolor{granat}{HTML}{677081}
\definecolor{ciemnyblekit}{HTML}{91A1B8}

\definecolor{oliwkowy}{HTML}{627037}
\definecolor{ciemnazielen}{HTML}{394D2E}
\definecolor{ciemnyfiolet}{HTML}{424444}
\definecolor{mocnyfiolet}{HTML}{717299}
\definecolor{jasnyfiolet}{HTML}{B0ABCC}
\definecolor{bladyfiolet}{HTML}{C9C7DB}

\definecolor{lightblue}{RGB}{135,206,250}
\definecolor{darkblue}{RGB}{0,0,160}
\definecolor{darkgreen}{RGB}{0,160,0}

\begin{document}

\title
{
Exceptional sequences of 8 line bundles
on $(\PP^1)^3$
}

\author[K.~Altmann]{Klaus Altmann
}
\address{Institut f\"ur Mathematik,
FU Berlin,
K\"onigin-Luise-Str.~24-26,
D-14195 Berlin,
Germany}
\email{altmann@math.fu-berlin.de}

\author[M.~Altmann]{Martin Altmann
}
\address{Mathematikwettbewerb K\"anguru e.V.,
c/o HU Berlin,
Institut f\"ur Mathematik,
Rudower Chaussee 25,
D-12489 Berlin,
Germany}
\email{maltmann@math.hu-berlin.de}

\begin{abstract}
We investigate \mes s of line bundles
on $(\PP^1)^r$, i.e.\ those consisting of
$2^r$ elements. For $r=3$ we show that they are always full, meaning
that they generate the derived category.
\\
Everything is done in the discrete setup: Exceptional sequences of
line bundles appear as special finite subsets $\cl$ of the {\scc Picard}
group $\Z^r$ of $(\PP^1)^r$, and the question of generation is understood
like a process of contamination of the whole $\Z^r$ out of
an infectious seed $\cl$.
\end{abstract}

\maketitle
\setcounter{tocdepth}{1}

\vspace*{-2ex}

\section{Introduction} \label{intro}

The content of the paper can be understood in two different languages.
While the version of Subsection~(\ref{combLang}) presents everything as a
challenging self-contained, combinatorial task, like a game to play,
the background motivation stems
from the algebro geometric scenario explained in Subsection~(\ref{algeomLang}).
The paper sticks to the first language. That is,
beyond Subsection~(\ref{basicDefP}), no algebraic geometry will appear.

\subsection{The combinatorial language}
\label{combLang}
Let $\cl=\{\cl^0,\ldots,\cl^{m-1}\}\subset\Z^3$
be an ordered subset such that for each $i<j$ there is
an index $\nu=\nu(i,j)\in\{1,2,3\}$ satisfying
$\cl^j_\nu-\cl^i_\nu=1$.
Then, beginning with $\cl$, we start a contamination procedure by declaring
each affine line $\ell\subset\Z^3$ parallel to a coordinate axis
to be infected if it contains at least two adjacent infected points.
See the Definitions \ref{def-fesProdPd} and \ref{def-contamination} for more
details.
Then, Theorem~\ref{th-bbbCont} states that $\cl$ consists of at most
eight elements and, moreover, if $|\cl|$ equals
this maximal number, then it contaminates the whole lattice $\Z^3$.

\begin{figure}[ht]
$
\begin{tikzpicture}[scale=0.4]
\begin{scope}[shift={(0,0)}]
\multido{\n=0+2}{6}{
\multido{\na=1+1}{5}{\draw[thin,color=black!40] (\na,\n) -- ++(5/3,5/3) (1/3*\na,1/3*\na+\n) -- ++(5,0);}
\draw[color=black] (0,\n) -- ++(5,0) -- ++(5/3,5/3) -- ++(-5,0) -- cycle;
}
\fill[color=coverH] (0+0/3,0/3+2*0) circle (3pt);
\fill[color=coverH] (0+0/3,0/3+2*1) circle (3pt);
\fill[color=coverH] (0+1/3,1/3+2*2) circle (3pt);
\fill[color=coverH] (4+1/3,1/3+2*3) circle (3pt);
\fill[color=coverH] (5+1/3,1/3+2*3) circle (3pt);
\fill[color=coverH] (1+4/3,4/3+2*4) circle (3pt);
\fill[color=coverH] (1+5/3,5/3+2*4) circle (3pt);
\fill[color=coverH] (1+2/3,2/3+2*5) circle (3pt);
\draw[thick,color=coverB] (0+0/3+1/4,2/3+2*0-0.1) node{$0$};
\draw[thick,color=coverB] (0+0/3+1/4,2/3+2*1-0.1) node{$1$};
\draw[thick,color=coverB] (0+1/3+1/4,3/3+2*2-0.1) node{$2$};
\draw[thick,color=coverB] (4+1/3+1/4,3/3+2*3-0.1) node{$3$};
\draw[thick,color=coverB] (5+1/3+1/4,3/3+2*3-0.1) node{$4$};
\draw[thick,color=coverB] (1+4/3-2/4,4/3+2*4-0.1) node{$5$};
\draw[thick,color=coverB] (1+5/3+1/4,7/3+2*4-0.1) node{$6$};
\draw[thick,color=coverB] (1+2/3+1/4,4/3+2*5-0.1) node{$7$};
\end{scope}

\begin{scope}[shift={(7.5,0)}]
\multido{\n=0+2}{6}{
\multido{\na=1+1}{5}{\draw[thin,color=black!40] (\na,\n) -- ++(5/3,5/3) (1/3*\na,1/3*\na+\n) -- ++(5,0);}
\draw[color=black] (0,\n) -- ++(5,0) -- ++(5/3,5/3) -- ++(-5,0) -- cycle;
}
\multido{\n=0+1}{6}{
\fill[color=coverH] ( 0+ 0/3, 0/3+2*\n) circle (3pt);
\fill[color=coverH] ( 1+\n/3,\n/3+2* 4) circle (3pt);
\fill[color=coverH] (\n+ 1/3, 1/3+2* 3) circle (3pt);
}
\fill[color=coverH] ( 0+ 1/3, 1/3+2* 2) circle (3pt);
\fill[color=coverH] ( 1+ 2/3, 2/3+2* 5) circle (3pt);
\fill[color=coverH] (0+0/3,0/3+2*0) circle (3pt);
\fill[color=coverG] (0+0/3,0/3+2*1) circle (3pt);
\fill[color=coverG] (0+1/3,1/3+2*2) circle (3pt);
\fill[color=coverG] (1+2/3,2/3+2*5) circle (3pt);
\fill[color=coverG] (1+4/3,4/3+2*4) circle (3pt);
\fill[color=coverG] (1+5/3,5/3+2*4) circle (3pt);
\fill[color=coverG] (4+1/3,1/3+2*3) circle (3pt);
\fill[color=coverG] (5+1/3,1/3+2*3) circle (3pt);
\end{scope}
\begin{scope}[shift={(15,0)}]
\multido{\n=0+2}{6}{
\multido{\na=1+1}{5}{\draw[thin,color=black!40] (\na,\n) -- ++(5/3,5/3) (1/3*\na,1/3*\na+\n) -- ++(5,0);}
\draw[color=black] (0,\n) -- ++(5,0) -- ++(5/3,5/3) -- ++(-5,0) -- cycle;
\multido{\n=0+1}{6}{
\fill[color=coverH] ( 0+ 0/3, 0/3+2*\n) circle (3pt);
\fill[color=coverH] ( 1+\n/3,\n/3+2* 4) circle (3pt);
\fill[color=coverH] (\n+ 1/3, 1/3+2* 3) circle (3pt);
\fill[color=coverH] ( 0+\n/3,\n/3+2* 2) circle (3pt);
\fill[color=coverH] ( 1+ 2/3, 2/3+2*\n) circle (3pt);
\fill[color=coverH] (\n+ 0/3, 0/3+2* 4) circle (3pt);
\fill[color=coverH] ( 0+\n/3,\n/3+2* 3) circle (3pt);
\fill[color=coverH] ( 0+ 1/3, 1/3+2*\n) circle (3pt);
\fill[color=coverH] ( 1+ 1/3, 1/3+2*\n) circle (3pt);
}
\multido{\n=0+1}{6}{
\fill[color=coverG] ( 0+ 0/3, 0/3+2*\n) circle (3pt);
\fill[color=coverG] ( 1+\n/3,\n/3+2* 4) circle (3pt);
\fill[color=coverG] (\n+ 1/3, 1/3+2* 3) circle (3pt);
}
\fill[color=coverG] ( 0+ 1/3, 1/3+2* 2) circle (3pt);
\fill[color=coverG] ( 1+ 2/3, 2/3+2* 5) circle (3pt);
}
\end{scope}
\begin{scope}[shift={(22.5,0)}]
\multido{\n=0+2}{6}{
\multido{\na=1+1}{5}{\draw[thin,color=black!40] (\na,\n) -- ++(5/3,5/3) (1/3*\na,1/3*\na+\n) -- ++(5,0);}
\draw[color=black] (0,\n) -- ++(5,0) -- ++(5/3,5/3) -- ++(-5,0) -- cycle;
\multido{\na=0+1}{6}{\multido{\nb=0+1}{6}{
\fill[color=coverH] (  0+\na/3,\na/3+2*\nb) circle (3pt);
\fill[color=coverH] (  1+\na/3,\na/3+2*\nb) circle (3pt);
\fill[color=coverH] (\na+  1/3,  1/3+2*\nb) circle (3pt);
}}
\multido{\n=2+1}{4}{
\fill[color=coverH] (\n+ 0/3, 0/3+2* 4) circle (3pt);
\fill[color=coverH] (\n+ 2/3, 2/3+2* 2) circle (3pt);
\fill[color=coverH] (\n+ 2/3, 2/3+2* 3) circle (3pt);
}
\multido{\n=0+1}{6}{
\fill[color=coverG] ( 0+ 0/3, 0/3+2*\n) circle (3pt);
\fill[color=coverG] ( 1+\n/3,\n/3+2* 4) circle (3pt);
\fill[color=coverG] (\n+ 1/3, 1/3+2* 3) circle (3pt);
\fill[color=coverG] ( 0+\n/3,\n/3+2* 2) circle (3pt);
\fill[color=coverG] ( 1+ 2/3, 2/3+2*\n) circle (3pt);
\fill[color=coverG] (\n+ 0/3, 0/3+2* 4) circle (3pt);
\fill[color=coverG] ( 0+\n/3,\n/3+2* 3) circle (3pt);
\fill[color=coverG] ( 0+ 1/3, 1/3+2*\n) circle (3pt);
\fill[color=coverG] ( 1+ 1/3, 1/3+2*\n) circle (3pt);
}
}
\end{scope}
\begin{scope}[shift={(30,0)}]
\multido{\n=0+2}{6}{
\multido{\na=1+1}{5}{\draw[thin,color=black!40] (\na,\n) -- ++(5/3,5/3) (1/3*\na,1/3*\na+\n) -- ++(5,0);}
\draw[color=black] (0,\n) -- ++(5,0) -- ++(5/3,5/3) -- ++(-5,0) -- cycle;
\multido{\na=2+1}{4}{\multido{\nb=0+1}{6}{\multido{\nc=0+1}{6}{
\fill[color=coverH] (\na+\nb/3,\nb/3+2*\nc) circle (3pt);
}}}
\multido{\na=0+1}{6}{\multido{\nb=0+1}{6}{
\fill[color=coverG] (  0+\na/3,\na/3+2*\nb) circle (3pt);
\fill[color=coverG] (  1+\na/3,\na/3+2*\nb) circle (3pt);
\fill[color=coverG] (\na+  1/3,  1/3+2*\nb) circle (3pt);
}}
\multido{\n=2+1}{4}{
\fill[color=coverG] (\n+ 0/3, 0/3+2* 4) circle (3pt);
\fill[color=coverG] (\n+ 2/3, 2/3+2* 2) circle (3pt);
\fill[color=coverG] (\n+ 2/3, 2/3+2* 3) circle (3pt);
}
}
\end{scope}
\end{tikzpicture}
$
\caption{An infection spreading, contaminating $\Z^3$ in four steps}
\label{fig:fourSteps}
\end{figure}
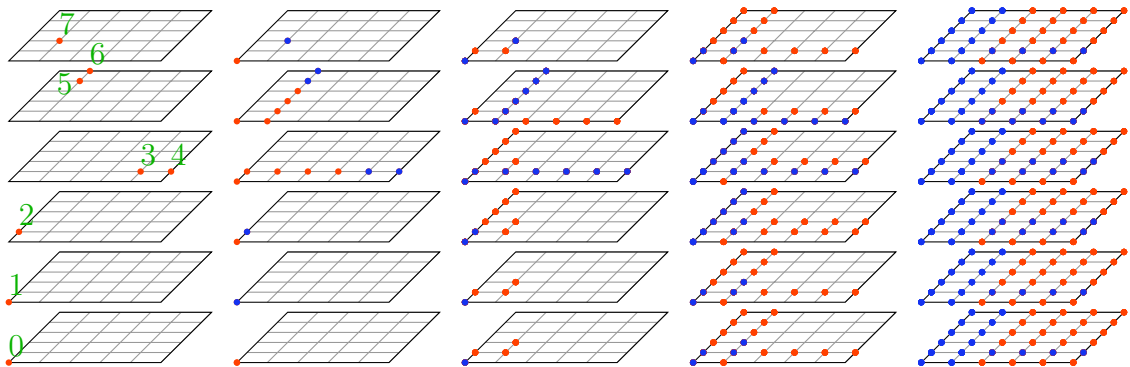

\subsection{The algebro geometric language}
\label{algeomLang}
To investigate the derived category of smooth, projective
algebraic varieties $X$
one tries to mimic the methods of linear algebra by working with
semiorthogonal decompositions and, more special,
with so-called exceptional sequences,
cf.~Definition~\ref{def-fes}. The latter can be
understood as an analog to linearly independent sets.
Moreover, by definition,
\fes s generate the entire derived category $\CD(X)$ so they correspond to
bases of vector spaces in linear algebra.
\\[1ex]
However, this comparison has a flaw: While the cardinality of
\fes s $\cl$, if they exist at all, is known to be the
rank of the {\scc Grothendieck} group $K_0(X)$,
it is not clear whether exceptional sequences $s$
with $|\cl|=\rank K_0(X)$ are automatically full.
In \cite{Phantom}, this problem was related to the existence of
so-called phantom
categories, i.e.\ non-trivial triangulated categories with vanishing
Hochschild homology and trivial {\scc Grothendieck} group. They may
appear as orthogonal
complements of those exceptional sequences. See \cite{summary_Phantoms}
and the citations therein for examples and a general discussion of this
subject.
\\[1ex]
A special situation for those questions appears with
the class of smooth, projective toric varieties
$X=\toric(\Sigma)$ for fans $\Sigma$ in some
real vector space $N_\R=\R^n$. Here, the rank of $K_0(X)$ equals
the number of full-dimensional cones inside $\Sigma$. Alternatively,
if $\Sigma$ appears as the normal fan $\normal(\Delta)$ of a smooth
lattice polytope $\Delta$ in the dual space $M_\R=(\R^n)^*$,
then $\rank K_0(X)$ equals the number of vertices of $\Delta$.
If one drops the requirement of line bundles and asks for general complexes
of coherent sheaves as elements of \fes s
instead, then their existence was guaranteed by {\scc Kawamata}'s papers
\cite{kaw1,kaw2,kaw3}. On the other hand, if one insists on line bundles,
then {\scc Efimov} has shown in \cite{efimov} that full exceptional sequences
cannot exist for all smooth, projective toric varieties.
\\[1ex]
In the present paper we do not address the question for which fans $\Sigma$
those sequences exist at all. Instead, we
consider the very special situation of $X=(\PP^1)^r$ where $\Delta$ is the
$r$-dimensional cube having $2^r$ vertices.
Here, the existence is guaranteed by the trivial
example $\cl=\{\CO(a)\kst a\in\{0,1\}^r\}$. However, it is not trivial at all 
whether all other exceptional sequences consisting of $2^r$ line bundles
generate $\Pic(\PP^1)^r=\Z^r$ which is generating $\CD(\PP^1)^r$.
This was also the main problem in \cite[(4.2)]{Mironov}.
For $r=2$ the question is rather trivial and will be discussed in
Subsection~(\ref{squareCase}).
Our main result is an affirmative answer for the case $r=3$:

\begin{theorem}
\label{th-bbb}
Every \mes\ of line bundles on $(\PP^1)^3$, i.e.\ consisting of eight
elements, is full.
\end{theorem}

This theorem appears later, as Theorem~\ref{th-bbbCont}
in Subsection~(\ref{Pad}),
in a slightly different form. There it is formulated in the combinatorial
manner as it was already announced in Subsection~(\ref{combLang}) --
that is, making use of the language of contaminations which will
be introduced in the
Subsections (\ref{specialProdProj}) and (\ref{contamProc}).
It should be mentioned that even in this very special case of $r=3$
we depend on the usage of computers.

\subsection{Previous work}
\label{prevWork}
Related work into the direction of the above theorem is contained in
the recent preprint \cite{pic2Fano} dealing with the case of
toric {\scc Fano} varieties of {\scc Picard} rank two and dimensions $3$ or $4$.
This is going to be generalized to arbitrary smooth, projective
toric varieties of {\scc Picard} rank two in \cite{fesFano}. The latter contains the
case of $X=\PP^d\times\PP^e$, e.g., $X=(\PP^1)^2$.
The result of the present paper is the first
step into the direction of a higher {\scc Picard} rank.
\\[1ex]
{\em Acknowledgments:} We would like to thank Lutz Hille pointing out to us
that the lack of phantom categories on $(\PP^1)^3$ was still a problem.
Moreover, we thank Christian Haase for encouraging us to publish this note
-- despite the fact that we were not able to solve the problem
in a pure way, i.e.~without the usage of a small computer search at the end.
Finally, we would like to thank Frederik Witt for many discussions
about this subject.

\section{The basic setup} \label{basicSetup}

\subsection{Basic definitions}\label{basicDefP}
Let us recall the basic definitions. We restrict ourselves to the case of line
bundles.

\begin{definition}
\label{def-fes}
1) A sequence $\cl=[\CL_0,\ldots,\CL_{m-1}]$ of line bundles on a smooth,
projective variety $X$ is called {\em exceptional} if
there are no backwards homomorphisms, i.e.~for each $i<j$ we have
$\gHom_{\CD(X)}(\CL_{j},\CL_i[*])=0$.
Here we denote by $\CD(X):=\CD^b(X)$ the bounded derived category, and
$*$ refers to arbitrary shifts.
In other words,
$\gH^k(X,\CL_i\otimes\CL_j^{-1})=0$ for all $k\in\Z$.
\\[1ex]
2) An exceptional sequence $\cl$ is called {\em maximal}
if $m=|\cl|=n:=\rank K_0(X)$.
\\[1ex]
3) An exceptional sequence $\cl$ is called {\em full} if the set
$\{\CL_0,\ldots,\CL_{m-1}\}$ generates the derived category $\CD(X)$. That is, the
latter is the smallest triangulated category containing these sheaves.
\end{definition}

Note that, in contrast to (1),
the properties (2) and (3) of the previous definition
do not depend on the particular
ordering within the sequence $\cl$.
For smooth, projective toric varieties $X=\toric(\Sigma)$,
we can identify $\Cl(X)=\Z^r$
with $r=\#\Sigma(1)-\dim X$.
Then, the classes $\cl^i=[\CL_i]$ correspond to certain points in $\Z^r$,
and the exceptionality condition asks for
\rule{0pt}{16pt}$\vect{\cl^i\,\cl^{j}}=\cl^{j}-\cl^i\in -\Imm(X)$ for $i<j$
where $\Imm(X)\subseteq\Z^r$ denotes the {\em immaculate locus}, i.e.\
the set of those classes of line bundles carrying no cohomology at all,
including the $0$-th one.
See \cite{immaculate} for a discussion of this interesting region.

\subsection{Products of projective spaces}
\label{specialProdProj}
Assume now that $X=\PP^{d_1}\times\ldots\times\PP^{d_r}$.
In contrast to arbitrary smooth, projective toric varieties,
this special case provides a very explicit and clear description of the
immaculate locus.
Since the invertible immaculate sheaves on $\PP^d$ are
$\CO(-1), \CO(-2), \ldots, \CO(-d)$, we know that
$$
-\Imm(X)=\{a\in\Z^r\kst 1\leq a_\nu\leq d_\nu\;\mbox{for some }
\nu=1,\ldots,r\}.
$$
Thus, everything
can be described in a purely combinatorial language. According to this,
and besides a short motivating remark at the beginning of
Subsection~(\ref{contamProc}),
\begin{center}
{\em no algebraic geometry will appear beyond this point}.
\end{center}
For any ordered subset $\cl\subseteq\Z^r$ we will
encode the particular position
of its elements by the upper index, i.e.\
$\cl=[\cl^0,\ldots,\cl^{m-1}]$. In contrast, the coordinates of a single
$\cl^i=(\cl^i_1,\ldots,\cl^i_r)\in\Z^r$ are indicated by lower indices.
Now, the content of Definition~\ref{def-fes} can be rewritten as follows:

\begin{definition}
\label{def-fesProdPd}
Assume that we are given a dimension vector
$d=(d_1,\ldots,d_r)\in\N^r$ with $r\geq 1$ and $d_\nu\geq 1$
for $\nu=1,\ldots,r$.
\\[1ex]
1) A sequence $\cl=[\cl^0,\ldots,\cl^{m-1}]$ in $\Z^r$ is called
{\em $d$-exceptional} if, for each $i<j$, there is an index
$\nu=\nu(i,j)\in\{1,\ldots,r\}$ such that
$$
(\cl^j-\cl^i)_\nu=\cl^j_\nu-\cl^i_\nu \in\{1,2,\ldots,d_\nu\}.
$$
2) A $d$-exceptional sequence $\cl$ is called {\em maximal}
if $m=|\cl|=n(d):=\prod_{\nu=1}^r(d_\nu+1)$.
\\[1ex]
3) A $d$-exceptional sequence $\cl$ is called {\em full} if
it $d$-contaminates the whole lattice $\Z^r$ where the latter notion will
be explained in the upcoming Definition~\ref{def-contamination}.
\end{definition}

\subsection{The contamination procedure}
\label{contamProc}
The following definition of the $d$-conta\-mi\-na\-tion process arises from
the {\scc Koszul} complex on $\PP^d$
$$
\xymatrixcolsep{1.8em}
\xymatrix{
0 \ar[r] &
\CO \ar[r] \ar@{=}[d] &
\CO(1)^{d+1} \ar[r] \ar@{=}[d] &
\CO(2)^{d+1\choose 2} \ar[r] \ar@{=}[d] &
\ldots \ar[r] &
\CO(d+1) \ar[r] \ar@{=}[d] &
0
\\
0 \ar[r] &
\Lambda^0\CO(1)^{d+1} \ar[r] &
\Lambda^1\CO(1)^{d+1} \ar[r] &
\Lambda^2\CO(1)^{d+1} \ar[r] &
\ldots \ar[r] &
\Lambda^{d+1}\CO(1)^{d+1} \ar[r] &
0
}
$$
showing that
the successive sheaves
$\CO, \CO(1), \ldots, \CO(d)$ generate $\CO(d+1)$ and, eventually, the
whole line $\Pic\PP^d=\Z$.

\begin{definition}
\label{def-contamination}
Let $d\in\N^r$ and $S\subseteq\Z^r$ an arbitrary subset.
Then, with $e_\nu$ denoting the $\nu$-th canonical basis vector of $\Z^r$,
the elements of the set
$$
\textstyle
\cont(S)\;:=\; S \;\cup\;
\Big(\bigcup_{\nu=1}^r \bigcup_{p,\,p+e_\nu,\ldots,\,p+d_\nu e_\nu\,\in\, S}
\,
(p+\Z \cdot e_\nu)\Big) \;\subseteq\;\Z^r
$$
are called {\em directly $d$-contaminated} from $S$.
This gives rise to the inductive spreading
$\cont^0(S):= S$ and
$$
\textstyle
\cont^k(S):=\cont\big(\cont^{k-1}(S)\big)
\hspace{0.5em}\mbox{for }\;k\geq 1
$$
which contain the results of $d$-contamination in at most $k$ steps.
Finally, the elements of
the overall union $\cont^\infty(S):=\bigcup_{k\geq 0} \cont^k(S)$
are simply called {\em $d$-contaminated}.
\end{definition}

\subsection{The special case $(\PP^1)^r$} \label{Pad}\label{adaptSpecial}
Now, we focus on the special case $d=(1,\ldots,1)\in\N^r$.
The notion of $d$-exceptionality will simply be called {\em exceptionality} then.
By Definition~\ref{def-fesProdPd},
it means that, for each $i<j$, there is an index
$\nu=\nu(i,j)\in\{1,\ldots,r\}$ such that
$$
(\cl^j-\cl^i)_\nu=\cl^j_\nu-\cl^i_\nu =1.
$$
The size of a maximal exceptional sequence is $n=2^r$.
The notion of $d$-contamination from
Subsection~(\ref{specialProdProj}) will simply be called {\em contamination}.
The essential part of Definition~\ref{def-contamination}, i.e.\
$$
\textstyle
\cont(S)\;=\; S \;\cup\;
\Big(\bigcup_{\nu=1}^r \bigcup_{p,\,p+e_\nu\in S} \,(p+\Z \cdot e_\nu)\Big)
\;\subseteq\;\Z^r
$$
does now just mean that two
adjacent lattice points of $\Z^r$ infect the whole affine line they span,
cf.~Figure~\ref{fig:fourSteps} before.
And here is our main result using the contamination language; it is equivalent
to Theorem~\ref{th-bbb} from the introduction:

\begin{theorem}
\label{th-bbbCont}
Let $r=3$. Then every $(1,1,1)$-exceptional sequence $\cl\subset \Z^3$
consists of at most eight elements. Moreover, if $|\cl|$ equals
this maximal number, then $\cl$ contaminates the whole lattice $\Z^3$.
\end{theorem}

The proof is contained in Section~\ref{compRThree}.

\subsection{The width of exceptional sequences}
\label{widthFES}
We conclude this setup section with another basic notion.
For any finite subset $S\subseteq\Z^r$ 
and $\nu\in\{1,\ldots,r\}$,
we define the $\nu$-width of $S$
as
$$
\width_\nu(S):=\max_{s\in S} s_\nu - \min_{s\in S} s_\nu +1.
$$
That is, $\width_\nu(S)$ is the smallest number $w\in\N$ such that $S$ fits
into
$w$ consecutive layers, i.e.\ integral hyperplanes
$[x_\nu=\mbox{const}]$ in $\nu$-direction.
In \cite{fesFano} it was shown that (maximal) exceptional sequences
for $X=\PP^{d_1}\times\PP^{d_2}$, i.e.~in the case $r=2$
of the setting of Subsection~(\ref{specialProdProj}),
any exceptional sequence
$\cl$ satisfies either $\width_1(\cl)\leq 2d_1+1$ or
$\width_2(\cl)\leq 2d_2+1$. This was an essential step for proving that
maximal
sequences are always full. However, as we already have seen in
Figure~\ref{fig:fourSteps} and will see in
Example~\ref{ex-cubeFourFourFour}, this simple kind of bounds do not
stay true for $r\geq 3$.

\section{Examples}\label{firstEx}

\subsection{The case of $r=1$}\label{lineCase}
There is not much to say in this case -- \mes s consist of two adjacent points
in $\Z^1$. And they contaminate the whole line immediately.

\subsection{The case of $r=2$}\label{squareCase}
Up to a possible switch of the coordinates,
there are two principal types of maximal exceptional
configurations -- they are shown in Figure~\ref{fig:planeSeq}.
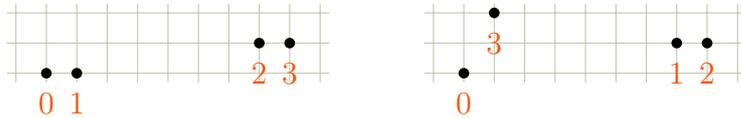
\begin{figure}[ht]
$
\begin{tikzpicture}[scale=0.4]
\draw[color=oliwkowy!40] (-2.3,-1.3) grid (8.3,1.3);
\fill[thick, color=black]
 (-1,-1) circle (5pt) (0,-1) circle (5pt)
 (6,0) circle (5pt) (7,0) circle (5pt);
\draw[thick, color=intOrange]
 (-1,-2) node{$0$} (0,-2) node{$1$} (6,-1) node{$2$} (7,-1) node{$3$};
\end{tikzpicture}
\hspace{3em}
\begin{tikzpicture}[scale=0.4]
\draw[color=oliwkowy!40] (-2.3,-1.3) grid (8.3,1.3);
\fill[thick, color=black]
 (-1,-1) circle (5pt) (0,1) circle (5pt)
 (6,0) circle (5pt) (7,0) circle (5pt);
\draw[thick, color=intOrange]
 (-1,-2) node{$0$} (0,0) node{$3$} (6,-1) node{$1$} (7,-1) node{$2$};
\end{tikzpicture}
\vspace{-1ex}
$
\caption{Plane exceptional sequences}
\label{fig:planeSeq}
\end{figure}
In both cases, the pair of points on the
central horizontal line $\ell$ can be arbitrarily shifted
along $\ell$. In any case,
the whole line $\ell$ becomes infected first, causing further
``vertical contaminations'' from the second pair of points.
\\[1ex]
In most cases, the numbering, i.e.\ the right
``exceptional ordering'' of the points
(indicated by the red labels $0,\ldots,3$) is
unique for both types -- but for a few shifts along $\ell$,
there remains a choice.
This classification result follows from simple combinatorial arguments --
or, alternatively, from the general theory developed later.
\\[1ex]
Note that the (unique) order of the right-hand example from Figure~\ref{fig:planeSeq} 
is horizontally lexicographic but not vertically lexicographic.

\subsection{The case of $r=3$}\label{cubeCase}
We are looking for configurations of 8 points in $\Z^3$.
The standard \fes\
corresponds to the 3-dimensional $(2\times 2\times 2)$-grid.
For this the contamination of the whole lattice $\Z^3$ is immediate.
However, there are more interesting and prettier examples.

\begin{example}
\label{ex-cubeFourFourFour}
A quite symmetric \mes\ is
depicted in Figure~\ref{fig:cubeTetrahedra}.
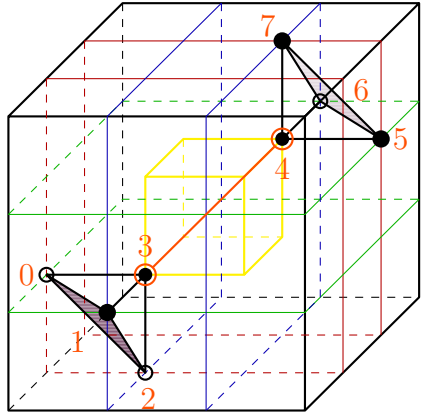
\begin{figure}[ht]
$
\begin{tikzpicture}[scale=1.3,join=round,
x={(1,0,0)},y={(0,1,0)},z={(0,0,-1)}]
\draw[thick, color=yellow]
 (1,1,1) -- (2,1,1) -- (2,2,1) -- (1,2,1) -- cycle
 (2,1,1) -- (2,1,2) -- (2,2,2) -- (2,2,1)
 (1,2,1) -- (1,2,2) -- (2,2,2);
\draw[dashed, color=yellow]
 (1,2,2) -- (1,1,2) -- (2,1,2) (1,1,1) -- (1,1,2);
\shade[left color=white!90!coverD, right color=coverD, shading angle=0,
       opacity=0.5]
 (0,1,1) -- (1,0,1) -- (1,1,0) -- cycle
 (3,2,2) -- (2,3,2) -- (2,2,3) -- cycle;
\draw[thick, color=black]
 (0,0,0) -- (3,0,0) -- (3,3,0) -- (0,3,0) -- cycle
 (3,0,0) -- (3,0,3) -- (3,3,3) -- (3,3,0)
 (0,3,0) -- (0,3,3) -- (3,3,3);
\draw[dashed, color=black]
 (0,3,3) -- (0,0,3) -- (3,0,3) (0,0,0) -- (0,0,3);
\draw[color=coverA]
 (3,0,1) -- (3,3,1) -- (0,3,1)
 (3,0,2) -- (3,3,2) -- (0,3,2);
\draw[dashed, color=coverA]
 (0,3,1) -- (0,0,1) -- (3,0,1)
 (0,3,2) -- (0,0,2) -- (3,0,2);
\draw[color=coverB]
 (0,1,0) -- (3,1,0) -- (3,1,3)
 (0,2,0) -- (3,2,0) -- (3,2,3);
\draw[dashed, color=coverB]
 (0,1,0) -- (0,1,3) -- (3,1,3)
 (0,2,0) -- (0,2,3) -- (3,2,3);
\draw[color=coverC]
 (1,0,0) -- (1,3,0) -- (1,3,3)
 (2,0,0) -- (2,3,0) -- (2,3,3);
\draw[dashed, color=coverC]
 (1,3,3) -- (1,0,3) -- (1,0,0)
 (2,3,3) -- (2,0,3) -- (2,0,0);
\fill[color=black]
 (1,1,1) circle (2pt) (2,2,2) circle (2pt)
 (1,1,0) circle (2.5pt) (3,2,2) circle (2.5pt) (2,3,2) circle (2.5pt);
\draw[thick, color=black]
 (0,1,1) circle (2pt) (1,0,1) circle (2pt) (2,2,3) circle (2pt);
\draw[thick, color=intOrange]
 (-0.2,1.0,1.0) node{$0$} (1.3,0.0,0.3) node{$2$} (0.7,0.7,0.0) node{$1$}
 (1.0,1.3,1.0) node{$3$} (2.0,1.7,2.0) node{$4$} (3.2,2.0,2.0) node{$5$}
 (2.3,2.0,3.3) node{$6$} (1.7,3.0,2.4) node{$7$};
\draw[thick, color=black]
 (1,1,1) -- (0,1,1) (1,1,1) -- (1,0,1) (1,1,1) -- (1,1,0)
 (0,1,1) -- (1,0,1) -- (1,1,0) -- cycle
 (2,2,2) -- (3,2,2) (2,2,2) -- (2,3,2) (2,2,2) -- (2,2,3)
 (3,2,2) -- (2,3,2) -- (2,2,3) -- cycle;
\draw[thick, color=intOrange]
 (1,1,1) circle (3pt) (2,2,2) circle (3pt);
\draw[thick, color=intOrange]
 (1.05,1.05,1.05) -- (1.95,1.95,1.95);
\end{tikzpicture}
$
\caption{Symmetric $4\times 4\times 4$ grid}
\label{fig:cubeTetrahedra}
\end{figure}
In the proposed numbering of the eight points of $\cl$
the order of the three vertices on both
of the tilted, regular triangles does not matter at all.
The only mandatory ordering that matters is
$$
\begin{array}{l}
[\mbox{triangle down left}] <
[\mbox{lower left red point}] <
\\
{}\hspace*{10em}
<[\mbox{upper right red point}] <
[\mbox{triangle up right}],
\end{array}
$$
that is, $\{\cl^0,\cl^1,\cl^2\} < \cl^3 < \cl^4 <\{\cl^5,\cl^6,\cl^7\}$.
The contamination process starts with the three lines through the
points $\cl^3$ and $\cl^4$, respectively. They infect all remaining six
vertices of the central yellow cube.
\end{example}

\begin{example}
\label{ex-cubeFiveFiveFive}
An even larger, but less symmetric example
inside a $5 \times 5 \times 5$ grid
is presented in Figure~\ref{fig:cubeFive}.
One possible numbering of the vertices is
$$
{\color{coverE} (1,0,0)}~
{\color{coverA} (0,1,1)}~
{\color{coverA} (1,3,1)}~
{\color{coverA} (1,4,1)}~
{\color{coverB} (3,1,2)}~
{\color{coverB} (4,1,2)}~
{\color{coverC} (2,2,3)}~
{\color{coverD} (2,2,4)}
$$
where the last coordinate indicates the height. The contamination process
begins with the central vertical line and a red and a green line in the layers
of height one and two, respectively.
\begin{figure}[ht]
$
\begin{tikzpicture}[scale=1.0,join=round,cap=round, x={(1,0,0)},y={(0,1,0)},z={(0,0,-1)}]
\draw[thick, color=black]
 (0,0,0) -- (4,0,0) -- (4,4,0) -- (0,4,0) -- cycle
 (4,0,0) -- (4,0,4) -- (4,4,4) -- (4,4,0)
 (0,4,0) -- (0,4,4) -- (4,4,4);
\draw[densely dashed, color=black] (0,0,4) -- (0,4,4) (0,0,4) -- (4,0,4) (0,0,4) -- (0,0,0);
\draw[color=black] (4,0,1) -- (4,4,1) -- (0,4,1) (4,0,2) -- (4,4,2) -- (0,4,2) (4,0,3) -- (4,4,3) -- (0,4,3);
\draw[color=black] (1,0,0) -- (1,4,0) -- (1,4,4) (2,0,0) -- (2,4,0) -- (2,4,4) (3,0,0) -- (3,4,0) -- (3,4,4);
\draw[densely dashed, color=black] (0,0,1) -- (0,4,1) (0,0,1) -- (4,0,1) (0,0,2) -- (0,4,2) (0,0,2) -- (4,0,2) (0,0,3) -- (0,4,3) (0,0,3) -- (4,0,3);
\draw[densely dashed, color=black] (1,0,4) -- (1,4,4) (1,0,4) -- (1,0,0) (2,0,4) -- (2,4,4) (2,0,4) -- (2,0,0) (3,0,4) -- (3,4,4) (3,0,4) -- (3,0,0);
\draw[color=coverA] (0,1,0) -- (4,1,0) -- (4,1,4);
\draw[color=coverB] (0,2,0) -- (4,2,0) -- (4,2,4);
\draw[color=coverC] (0,3,0) -- (4,3,0) -- (4,3,4);
\draw[densely dashed, color=coverA] (0,1,4) -- (0,1,0) (0,1,4) -- (4,1,4);
\draw[densely dashed, color=coverB] (0,2,4) -- (0,2,0) (0,2,4) -- (4,2,4);
\draw[densely dashed, color=coverC] (0,3,4) -- (0,3,0) (0,3,4) -- (4,3,4);
\draw[thick, color=yellow] (1,0,3) -- (1,4,3) (3,0,1) -- (3,4,1) (2,0,2) -- (2,4,2);
\fill[color=coverA] (0,1,1) circle (3pt);
\fill[color=coverE] (1,0,0) circle (3pt);
\fill[color=coverA] (1,1,3) circle (3pt);
\fill[color=coverA] (1,1,4) circle (3pt);
\fill[color=coverB] (3,2,1) circle (3pt);
\fill[color=coverB] (4,2,1) circle (3pt);
\fill[color=coverC] (2,3,2) circle (3pt);
\fill[color=coverD] (2,4,2) circle (3pt);
\end{tikzpicture}
\hspace*{2cm}
\begin{tikzpicture}[scale=0.75]
\multido{\n=0.0+1.5}{5}{
\multido{\na=1+1}{4}{\draw[thin,color=black!40] (\na,\n) -- ++(4/3,4/3) (1/3*\na,1/3*\na+\n) -- ++(4,0);}
\draw[color=black] (0,\n) -- ++(4,0) -- ++(4/3,4/3) -- ++(-4,0) -- cycle;
}
\fill[color=coverA] (0+1/3*1,1/3*1+1.5*1) circle (3pt);
\fill[color=coverE] (1+1/3*0,1/3*0+1.5*0) circle (3pt);
\fill[color=coverA] (1+1/3*3,1/3*3+1.5*1) circle (3pt);
\fill[color=coverA] (1+1/3*4,1/3*4+1.5*1) circle (3pt);
\fill[color=coverB] (3+1/3*1,1/3*1+1.5*2) circle (3pt);
\fill[color=coverB] (4+1/3*1,1/3*1+1.5*2) circle (3pt);
\fill[color=coverC] (2+1/3*2,1/3*2+1.5*3) circle (3pt);
\fill[color=coverD] (2+1/3*2,1/3*2+1.5*4) circle (3pt);
\end{tikzpicture}
$
\caption{Spanning the $5\times 5\times 5$ grid}
\label{fig:cubeFive}
\end{figure}
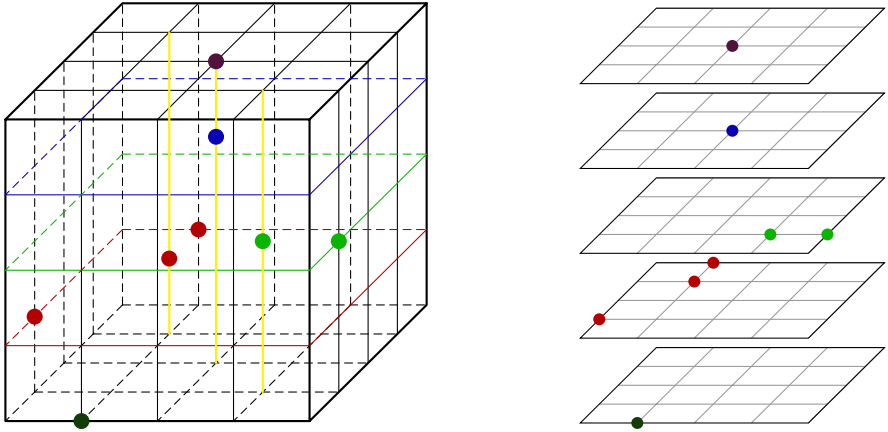
The whole process is illustrated in Figure~\ref{fig:cubeFiveInfect}.
Here, as we have already done in Figure~\ref{fig:fourSteps}, 
we have used bright red
color to mark recently infected points -- afterwards,
in the following steps, they will turn blue.
Note that the present example,
as it did Example~\ref{ex-cubeFourFourFour}, lacks
empty inner layers in any direction.
\end{example}

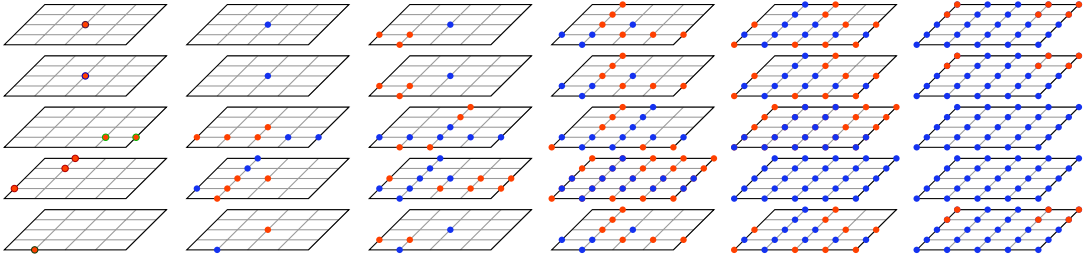
\begin{figure}[ht]
$
\begin{tikzpicture}[scale=0.4]
\begin{scope}[shift={(0,0)}]
\multido{\n=0.0+1.7}{5}{
\multido{\na=1+1}{4}{\draw[thin,color=black!40] (\na,\n) -- ++(4/3,4/3) (1/3*\na,1/3*\na+\n) -- ++(4,0);}
\draw[color=black] (0,\n) -- ++(4,0) -- ++(4/3,4/3) -- ++(-4,0) -- cycle;
}
\fill[color=coverH] (0+1/3*1,1/3*1+1.7*1) circle (3pt);
\fill[color=coverH] (1+1/3*3,1/3*3+1.7*1) circle (3pt);
\fill[color=coverH] (1+1/3*4,1/3*4+1.7*1) circle (3pt);
\fill[color=coverH] (3+1/3*1,1/3*1+1.7*2) circle (3pt);
\fill[color=coverH] (4+1/3*1,1/3*1+1.7*2) circle (3pt);
\fill[color=coverH] (2+1/3*2,1/3*2+1.7*3) circle (3pt);
\fill[color=coverH] (2+1/3*2,1/3*2+1.7*4) circle (3pt);
\fill[color=coverH] (1+1/3*0,1/3*0+1.7*0) circle (3pt);
\draw[color=coverA] (0+1/3*1,1/3*1+1.7*1) circle (3pt);
\draw[color=coverA] (1+1/3*3,1/3*3+1.7*1) circle (3pt);
\draw[color=coverA] (1+1/3*4,1/3*4+1.7*1) circle (3pt);
\draw[color=coverB] (3+1/3*1,1/3*1+1.7*2) circle (3pt);
\draw[color=coverB] (4+1/3*1,1/3*1+1.7*2) circle (3pt); 
\draw[color=coverC] (2+1/3*2,1/3*2+1.7*3) circle (3pt); 
\draw[color=coverD] (2+1/3*2,1/3*2+1.7*4) circle (3pt); 
\draw[color=coverE] (1+1/3*0,1/3*0+1.7*0) circle (3pt); 
\end{scope}
\begin{scope}[shift={(6,0)}]
\multido{\n=0.0+1.7}{5}{
\multido{\na=1+1}{4}{\draw[thin,color=black!40] (\na,\n) -- ++(4/3,4/3) (1/3*\na,1/3*\na+\n) -- ++(4,0);}
\draw[color=black] (0,\n) -- ++(4,0) -- ++(4/3,4/3) -- ++(-4,0) -- cycle;
}
\fill[color=coverG] (0+1/3*1,1/3*1+1.7*1) circle (3pt);
\fill[color=coverH] (0+1/3*1,1/3*1+1.7*2) circle (3pt);
\fill[color=coverG] (1+1/3*0,1/3*0+1.7*0) circle (3pt);
\fill[color=coverH] (1+1/3*0,1/3*0+1.7*1) circle (3pt);
\fill[color=coverH] (1+1/3*1,1/3*1+1.7*1) circle (3pt);
\fill[color=coverH] (1+1/3*1,1/3*1+1.7*2) circle (3pt);
\fill[color=coverH] (1+1/3*2,1/3*2+1.7*1) circle (3pt);
\fill[color=coverG] (1+1/3*3,1/3*3+1.7*1) circle (3pt);
\fill[color=coverG] (1+1/3*4,1/3*4+1.7*1) circle (3pt);
\fill[color=coverH] (2+1/3*1,1/3*1+1.7*2) circle (3pt);
\fill[color=coverH] (2+1/3*2,1/3*2+1.7*0) circle (3pt);
\fill[color=coverH] (2+1/3*2,1/3*2+1.7*1) circle (3pt);
\fill[color=coverH] (2+1/3*2,1/3*2+1.7*2) circle (3pt);
\fill[color=coverG] (2+1/3*2,1/3*2+1.7*3) circle (3pt);
\fill[color=coverG] (2+1/3*2,1/3*2+1.7*4) circle (3pt);
\fill[color=coverG] (3+1/3*1,1/3*1+1.7*2) circle (3pt);
\fill[color=coverG] (4+1/3*1,1/3*1+1.7*2) circle (3pt);
\end{scope}
\begin{scope}[shift={(12,0)}]
\multido{\n=0.0+1.7}{5}{
\multido{\na=1+1}{4}{\draw[thin,color=black!40] (\na,\n) -- ++(4/3,4/3) (1/3*\na,1/3*\na+\n) -- ++(4,0);}
\draw[color=black] (0,\n) -- ++(4,0) -- ++(4/3,4/3) -- ++(-4,0) -- cycle;
}
\fill[color=coverH] (0+1/3*1,1/3*1+1.7*0) circle (3pt);
\fill[color=coverG] (0+1/3*1,1/3*1+1.7*1) circle (3pt);
\fill[color=coverG] (0+1/3*1,1/3*1+1.7*2) circle (3pt);
\fill[color=coverH] (0+1/3*1,1/3*1+1.7*3) circle (3pt);
\fill[color=coverH] (0+1/3*1,1/3*1+1.7*4) circle (3pt);
\fill[color=coverH] (0+1/3*2,1/3*2+1.7*1) circle (3pt);
\fill[color=coverG] (1+1/3*0,1/3*0+1.7*0) circle (3pt);
\fill[color=coverG] (1+1/3*0,1/3*0+1.7*1) circle (3pt);
\fill[color=coverH] (1+1/3*0,1/3*0+1.7*2) circle (3pt);
\fill[color=coverH] (1+1/3*0,1/3*0+1.7*3) circle (3pt);
\fill[color=coverH] (1+1/3*0,1/3*0+1.7*4) circle (3pt);
\fill[color=coverH] (1+1/3*1,1/3*1+1.7*0) circle (3pt);
\fill[color=coverG] (1+1/3*1,1/3*1+1.7*1) circle (3pt);
\fill[color=coverG] (1+1/3*1,1/3*1+1.7*2) circle (3pt);
\fill[color=coverH] (1+1/3*1,1/3*1+1.7*3) circle (3pt);
\fill[color=coverH] (1+1/3*1,1/3*1+1.7*4) circle (3pt);
\fill[color=coverG] (1+1/3*2,1/3*2+1.7*1) circle (3pt);
\fill[color=coverG] (1+1/3*3,1/3*3+1.7*1) circle (3pt);
\fill[color=coverG] (1+1/3*4,1/3*4+1.7*1) circle (3pt);
\fill[color=coverH] (2+1/3*0,1/3*0+1.7*2) circle (3pt);
\fill[color=coverH] (2+1/3*1,1/3*1+1.7*1) circle (3pt);
\fill[color=coverG] (2+1/3*1,1/3*1+1.7*2) circle (3pt);
\fill[color=coverG] (2+1/3*2,1/3*2+1.7*0) circle (3pt);
\fill[color=coverG] (2+1/3*2,1/3*2+1.7*1) circle (3pt);
\fill[color=coverG] (2+1/3*2,1/3*2+1.7*2) circle (3pt);
\fill[color=coverG] (2+1/3*2,1/3*2+1.7*3) circle (3pt);
\fill[color=coverG] (2+1/3*2,1/3*2+1.7*4) circle (3pt);
\fill[color=coverH] (2+1/3*3,1/3*3+1.7*2) circle (3pt);
\fill[color=coverH] (2+1/3*4,1/3*4+1.7*2) circle (3pt);
\fill[color=coverH] (3+1/3*1,1/3*1+1.7*1) circle (3pt);
\fill[color=coverG] (3+1/3*1,1/3*1+1.7*2) circle (3pt);
\fill[color=coverH] (3+1/3*2,1/3*2+1.7*1) circle (3pt);
\fill[color=coverH] (4+1/3*1,1/3*1+1.7*1) circle (3pt);
\fill[color=coverG] (4+1/3*1,1/3*1+1.7*2) circle (3pt);
\fill[color=coverH] (4+1/3*2,1/3*2+1.7*1) circle (3pt);
\end{scope}
\begin{scope}[shift={(18,0)}]
\multido{\n=0.0+1.7}{5}{
\multido{\na=1+1}{4}{\draw[thin,color=black!40] (\na,\n) -- ++(4/3,4/3) (1/3*\na,1/3*\na+\n) -- ++(4,0);}
\draw[color=black] (0,\n) -- ++(4,0) -- ++(4/3,4/3) -- ++(-4,0) -- cycle;
}
\multido{\na=0+1}{5}{\multido{\nb=0+1}{5}{\fill[color=coverH] (\na+1/3*\nb,1/3*\nb+1.7*1) circle (3pt);}}
\fill[color=coverH] (0+1/3*0,1/3*0+1.7*2) circle (3pt);
\fill[color=coverG] (0+1/3*1,1/3*1+1.7*0) circle (3pt);
\fill[color=coverG] (0+1/3*1,1/3*1+1.7*1) circle (3pt);
\fill[color=coverG] (0+1/3*1,1/3*1+1.7*2) circle (3pt);
\fill[color=coverG] (0+1/3*1,1/3*1+1.7*3) circle (3pt);
\fill[color=coverG] (0+1/3*1,1/3*1+1.7*4) circle (3pt);
\fill[color=coverG] (0+1/3*2,1/3*2+1.7*1) circle (3pt);
\fill[color=coverG] (1+1/3*0,1/3*0+1.7*0) circle (3pt);
\fill[color=coverG] (1+1/3*0,1/3*0+1.7*1) circle (3pt);
\fill[color=coverG] (1+1/3*0,1/3*0+1.7*2) circle (3pt);
\fill[color=coverG] (1+1/3*0,1/3*0+1.7*3) circle (3pt);
\fill[color=coverG] (1+1/3*0,1/3*0+1.7*4) circle (3pt);
\fill[color=coverG] (1+1/3*1,1/3*1+1.7*0) circle (3pt);
\fill[color=coverG] (1+1/3*1,1/3*1+1.7*1) circle (3pt);
\fill[color=coverG] (1+1/3*1,1/3*1+1.7*2) circle (3pt);
\fill[color=coverG] (1+1/3*1,1/3*1+1.7*3) circle (3pt);
\fill[color=coverG] (1+1/3*1,1/3*1+1.7*4) circle (3pt);
\fill[color=coverH] (1+1/3*2,1/3*2+1.7*0) circle (3pt);
\fill[color=coverG] (1+1/3*2,1/3*2+1.7*1) circle (3pt);
\fill[color=coverH] (1+1/3*2,1/3*2+1.7*2) circle (3pt);
\fill[color=coverH] (1+1/3*2,1/3*2+1.7*3) circle (3pt);
\fill[color=coverH] (1+1/3*2,1/3*2+1.7*4) circle (3pt);
\fill[color=coverH] (1+1/3*3,1/3*3+1.7*0) circle (3pt);
\fill[color=coverG] (1+1/3*3,1/3*3+1.7*1) circle (3pt);
\fill[color=coverH] (1+1/3*3,1/3*3+1.7*2) circle (3pt);
\fill[color=coverH] (1+1/3*3,1/3*3+1.7*3) circle (3pt);
\fill[color=coverH] (1+1/3*3,1/3*3+1.7*4) circle (3pt);
\fill[color=coverH] (1+1/3*4,1/3*4+1.7*0) circle (3pt);
\fill[color=coverG] (1+1/3*4,1/3*4+1.7*1) circle (3pt);
\fill[color=coverH] (1+1/3*4,1/3*4+1.7*2) circle (3pt);
\fill[color=coverH] (1+1/3*4,1/3*4+1.7*3) circle (3pt);
\fill[color=coverH] (1+1/3*4,1/3*4+1.7*4) circle (3pt);
\fill[color=coverG] (2+1/3*0,1/3*0+1.7*2) circle (3pt);
\fill[color=coverH] (2+1/3*1,1/3*1+1.7*0) circle (3pt);
\fill[color=coverG] (2+1/3*1,1/3*1+1.7*1) circle (3pt);
\fill[color=coverG] (2+1/3*1,1/3*1+1.7*2) circle (3pt);
\fill[color=coverH] (2+1/3*1,1/3*1+1.7*3) circle (3pt);
\fill[color=coverH] (2+1/3*1,1/3*1+1.7*4) circle (3pt);
\fill[color=coverG] (2+1/3*2,1/3*2+1.7*0) circle (3pt);
\fill[color=coverG] (2+1/3*2,1/3*2+1.7*1) circle (3pt);
\fill[color=coverG] (2+1/3*2,1/3*2+1.7*2) circle (3pt);
\fill[color=coverG] (2+1/3*2,1/3*2+1.7*3) circle (3pt);
\fill[color=coverG] (2+1/3*2,1/3*2+1.7*4) circle (3pt);
\fill[color=coverG] (2+1/3*3,1/3*3+1.7*2) circle (3pt);
\fill[color=coverG] (2+1/3*4,1/3*4+1.7*2) circle (3pt);
\fill[color=coverH] (3+1/3*0,1/3*0+1.7*2) circle (3pt);
\fill[color=coverH] (3+1/3*1,1/3*1+1.7*0) circle (3pt);
\fill[color=coverG] (3+1/3*1,1/3*1+1.7*1) circle (3pt);
\fill[color=coverG] (3+1/3*1,1/3*1+1.7*2) circle (3pt);
\fill[color=coverH] (3+1/3*1,1/3*1+1.7*3) circle (3pt);
\fill[color=coverH] (3+1/3*1,1/3*1+1.7*4) circle (3pt);
\fill[color=coverG] (3+1/3*2,1/3*2+1.7*1) circle (3pt);
\fill[color=coverH] (4+1/3*0,1/3*0+1.7*2) circle (3pt);
\fill[color=coverH] (4+1/3*1,1/3*1+1.7*0) circle (3pt);
\fill[color=coverG] (4+1/3*1,1/3*1+1.7*1) circle (3pt);
\fill[color=coverG] (4+1/3*1,1/3*1+1.7*2) circle (3pt);
\fill[color=coverH] (4+1/3*1,1/3*1+1.7*3) circle (3pt);
\fill[color=coverH] (4+1/3*1,1/3*1+1.7*4) circle (3pt);
\fill[color=coverG] (4+1/3*2,1/3*2+1.7*1) circle (3pt);
\end{scope}
\begin{scope}[shift={(24,0)}]
\multido{\n=0.0+1.7}{5}{
\multido{\na=1+1}{4}{\draw[thin,color=black!40] (\na,\n) -- ++(4/3,4/3) (1/3*\na,1/3*\na+\n) -- ++(4,0);}
\draw[color=black] (0,\n) -- ++(4,0) -- ++(4/3,4/3) -- ++(-4,0) -- cycle;
}
\multido{\na=0+1}{5}{\multido{\nb=0+1}{5}{\fill[color=coverG] (\na+1/3*\nb,1/3*\nb+1.7*1) circle (3pt);}}
\multido{\na=0+1}{5}{\multido{\nb=0+1}{5}{\fill[color=coverH] (\na+1/3*\nb,1/3*\nb+1.7*2) circle (3pt);}}
\fill[color=coverH] (0+1/3*0,1/3*0+1.7*0) circle (3pt);
\fill[color=coverG] (0+1/3*0,1/3*0+1.7*2) circle (3pt);
\fill[color=coverH] (0+1/3*0,1/3*0+1.7*3) circle (3pt);
\fill[color=coverH] (0+1/3*0,1/3*0+1.7*4) circle (3pt);
\fill[color=coverG] (0+1/3*1,1/3*1+1.7*0) circle (3pt);
\fill[color=coverG] (0+1/3*1,1/3*1+1.7*2) circle (3pt);
\fill[color=coverG] (0+1/3*1,1/3*1+1.7*3) circle (3pt);
\fill[color=coverG] (0+1/3*1,1/3*1+1.7*4) circle (3pt);
\fill[color=coverH] (0+1/3*2,1/3*2+1.7*0) circle (3pt);
\fill[color=coverH] (0+1/3*2,1/3*2+1.7*3) circle (3pt);
\fill[color=coverH] (0+1/3*2,1/3*2+1.7*4) circle (3pt);
\fill[color=coverG] (1+1/3*0,1/3*0+1.7*0) circle (3pt);
\fill[color=coverG] (1+1/3*0,1/3*0+1.7*2) circle (3pt);
\fill[color=coverG] (1+1/3*0,1/3*0+1.7*3) circle (3pt);
\fill[color=coverG] (1+1/3*0,1/3*0+1.7*4) circle (3pt);
\fill[color=coverG] (1+1/3*1,1/3*1+1.7*0) circle (3pt);
\fill[color=coverG] (1+1/3*1,1/3*1+1.7*2) circle (3pt);
\fill[color=coverG] (1+1/3*1,1/3*1+1.7*3) circle (3pt);
\fill[color=coverG] (1+1/3*1,1/3*1+1.7*4) circle (3pt);
\fill[color=coverG] (1+1/3*2,1/3*2+1.7*0) circle (3pt);
\fill[color=coverG] (1+1/3*2,1/3*2+1.7*2) circle (3pt);
\fill[color=coverG] (1+1/3*2,1/3*2+1.7*3) circle (3pt);
\fill[color=coverG] (1+1/3*2,1/3*2+1.7*4) circle (3pt);
\fill[color=coverG] (1+1/3*3,1/3*3+1.7*0) circle (3pt);
\fill[color=coverG] (1+1/3*3,1/3*3+1.7*2) circle (3pt);
\fill[color=coverG] (1+1/3*3,1/3*3+1.7*3) circle (3pt);
\fill[color=coverG] (1+1/3*3,1/3*3+1.7*4) circle (3pt);
\fill[color=coverG] (1+1/3*4,1/3*4+1.7*0) circle (3pt);
\fill[color=coverG] (1+1/3*4,1/3*4+1.7*2) circle (3pt);
\fill[color=coverG] (1+1/3*4,1/3*4+1.7*3) circle (3pt);
\fill[color=coverG] (1+1/3*4,1/3*4+1.7*4) circle (3pt);
\fill[color=coverH] (2+1/3*0,1/3*0+1.7*0) circle (3pt);
\fill[color=coverG] (2+1/3*0,1/3*0+1.7*2) circle (3pt);
\fill[color=coverH] (2+1/3*0,1/3*0+1.7*3) circle (3pt);
\fill[color=coverH] (2+1/3*0,1/3*0+1.7*4) circle (3pt);
\fill[color=coverG] (2+1/3*1,1/3*1+1.7*0) circle (3pt);
\fill[color=coverG] (2+1/3*1,1/3*1+1.7*2) circle (3pt);
\fill[color=coverG] (2+1/3*1,1/3*1+1.7*3) circle (3pt);
\fill[color=coverG] (2+1/3*1,1/3*1+1.7*4) circle (3pt);
\fill[color=coverG] (2+1/3*2,1/3*2+1.7*0) circle (3pt);
\fill[color=coverG] (2+1/3*2,1/3*2+1.7*2) circle (3pt);
\fill[color=coverG] (2+1/3*2,1/3*2+1.7*3) circle (3pt);
\fill[color=coverG] (2+1/3*2,1/3*2+1.7*4) circle (3pt);
\fill[color=coverH] (2+1/3*3,1/3*3+1.7*0) circle (3pt);
\fill[color=coverG] (2+1/3*3,1/3*3+1.7*2) circle (3pt);
\fill[color=coverH] (2+1/3*3,1/3*3+1.7*3) circle (3pt);
\fill[color=coverH] (2+1/3*3,1/3*3+1.7*4) circle (3pt);
\fill[color=coverH] (2+1/3*4,1/3*4+1.7*0) circle (3pt);
\fill[color=coverG] (2+1/3*4,1/3*4+1.7*2) circle (3pt);
\fill[color=coverH] (2+1/3*4,1/3*4+1.7*3) circle (3pt);
\fill[color=coverH] (2+1/3*4,1/3*4+1.7*4) circle (3pt);
\fill[color=coverH] (3+1/3*0,1/3*0+1.7*0) circle (3pt);
\fill[color=coverG] (3+1/3*0,1/3*0+1.7*2) circle (3pt);
\fill[color=coverH] (3+1/3*0,1/3*0+1.7*3) circle (3pt);
\fill[color=coverH] (3+1/3*0,1/3*0+1.7*4) circle (3pt);
\fill[color=coverG] (3+1/3*1,1/3*1+1.7*0) circle (3pt);
\fill[color=coverG] (3+1/3*1,1/3*1+1.7*2) circle (3pt);
\fill[color=coverG] (3+1/3*1,1/3*1+1.7*3) circle (3pt);
\fill[color=coverG] (3+1/3*1,1/3*1+1.7*4) circle (3pt);
\fill[color=coverH] (3+1/3*2,1/3*2+1.7*0) circle (3pt);
\fill[color=coverH] (3+1/3*2,1/3*2+1.7*3) circle (3pt);
\fill[color=coverH] (3+1/3*2,1/3*2+1.7*4) circle (3pt);
\fill[color=coverH] (4+1/3*0,1/3*0+1.7*0) circle (3pt);
\fill[color=coverG] (4+1/3*0,1/3*0+1.7*2) circle (3pt);
\fill[color=coverH] (4+1/3*0,1/3*0+1.7*3) circle (3pt);
\fill[color=coverH] (4+1/3*0,1/3*0+1.7*4) circle (3pt);
\fill[color=coverG] (4+1/3*1,1/3*1+1.7*0) circle (3pt);
\fill[color=coverG] (4+1/3*1,1/3*1+1.7*2) circle (3pt);
\fill[color=coverG] (4+1/3*1,1/3*1+1.7*3) circle (3pt);
\fill[color=coverG] (4+1/3*1,1/3*1+1.7*4) circle (3pt);
\fill[color=coverH] (4+1/3*2,1/3*2+1.7*0) circle (3pt);
\fill[color=coverH] (4+1/3*2,1/3*2+1.7*3) circle (3pt);
\fill[color=coverH] (4+1/3*2,1/3*2+1.7*4) circle (3pt);
\end{scope}
\begin{scope}[shift={(30,0)}]
\multido{\n=0.0+1.7}{5}{
\multido{\na=1+1}{4}{\draw[thin,color=black!40] (\na,\n) -- ++(4/3,4/3) (1/3*\na,1/3*\na+\n) -- ++(4,0);}
\draw[color=black] (0,\n) -- ++(4,0) -- ++(4/3,4/3) -- ++(-4,0) -- cycle;
}
\multido{\na=0+1}{5}{\multido{\nb=0+1}{5}{\fill[color=coverG] (\na+1/3*\nb,1/3*\nb+1.7*1) circle (3pt);}}
\multido{\na=0+1}{5}{\multido{\nb=0+1}{5}{\fill[color=coverG] (\na+1/3*\nb,1/3*\nb+1.7*2) circle (3pt);}}
\multido{\na=0+1}{5}{\multido{\nb=0+1}{5}{\fill[color=coverG] (\na+1/3*\nb,1/3*\nb+1.7*3) circle (3pt);}}
\multido{\na=0+1}{5}{\multido{\nb=0+1}{5}{\fill[color=coverG] (\na+1/3*\nb,1/3*\nb+1.7*4) circle (3pt);}}
\multido{\na=0+1}{5}{\multido{\nb=0+1}{5}{\fill[color=coverG] (\na+1/3*\nb,1/3*\nb+1.7*0) circle (3pt);}}
\fill[color=coverH] (0+1/3*3,1/3*3+1.7*0) circle (3pt);
\fill[color=coverH] (0+1/3*3,1/3*3+1.7*3) circle (3pt);
\fill[color=coverH] (0+1/3*3,1/3*3+1.7*4) circle (3pt);
\fill[color=coverH] (0+1/3*4,1/3*4+1.7*0) circle (3pt);
\fill[color=coverH] (0+1/3*4,1/3*4+1.7*3) circle (3pt);
\fill[color=coverH] (0+1/3*4,1/3*4+1.7*4) circle (3pt);
\fill[color=coverH] (3+1/3*3,1/3*3+1.7*0) circle (3pt);
\fill[color=coverH] (3+1/3*3,1/3*3+1.7*3) circle (3pt);
\fill[color=coverH] (3+1/3*3,1/3*3+1.7*4) circle (3pt);
\fill[color=coverH] (3+1/3*4,1/3*4+1.7*0) circle (3pt);
\fill[color=coverH] (3+1/3*4,1/3*4+1.7*3) circle (3pt);
\fill[color=coverH] (3+1/3*4,1/3*4+1.7*4) circle (3pt);
\fill[color=coverH] (4+1/3*3,1/3*3+1.7*0) circle (3pt);
\fill[color=coverH] (4+1/3*3,1/3*3+1.7*3) circle (3pt);
\fill[color=coverH] (4+1/3*3,1/3*3+1.7*4) circle (3pt);
\fill[color=coverH] (4+1/3*4,1/3*4+1.7*0) circle (3pt);
\fill[color=coverH] (4+1/3*4,1/3*4+1.7*3) circle (3pt);
\fill[color=coverH] (4+1/3*4,1/3*4+1.7*4) circle (3pt);
\end{scope}
\end{tikzpicture}
$
\caption{Contaminating the whole $5\times 5\times 5$ grid (and $\Z^3$) in five steps}
\label{fig:cubeFiveInfect}
\end{figure}

\begin{example}
\label{ex-slowInfections}
According to Definition~\ref{def-contamination}, the contamination process
is an inductive procedure, i.e.\ it consists of a number of successive steps.
In Figure~\ref{fig:fourSteps}
we had depicted an example where the whole
lattice $\Z^3$ becomes contaminated in four steps. Its
starting \mes\ consisted of the points
$$
(0,0,0),\;(0,0,1),\;(0,1,2),\;(4,1,3),\;(5,1,3),\;(1,4,4),\;(1,5,4),\;(1,2,5).
$$
While, in Example~\ref{ex-cubeFiveFiveFive}, we had already presented a 
configuration
leading to a contamination process involving five steps, there are even
longer examples: The sequences
$$
\renewcommand{\arraystretch}{1.3}
\begin{array}{rcl}
&&[(0,0,0),(0,0,1),(0,1,4),(1,2,5),(1,4,4),(1,5,4),(4,1,3),(5,1,5)]
\hspace{1em}\mbox{and}
\\
&&[(0,0,0),(0,1,3),(1,0,4),(1,4,3),(1,5,5),(2,2,1),(4,1,4),(5,1,4)]
\end{array}
$$
require six or even seven steps, respectively.
\end{example}

\begin{example}
\label{ex-cubeNNN}
Allowing empty layers, as we have already done in
Example~\ref{ex-slowInfections}, leads to \mes s
that can stretch arbitrarily far in all directions.
For $n \in \N$ there is the maximal exceptional sequence
$$
(0,0,0)~(1,n,n)~(1,n,n+1)~(1,n+1,0)~(2,1,0)~(2,n,1)~(n,n+1,1)~(n+1,n+1,1)
$$
which stretches $n+1$ steps in every direction.
That is, in the language of Subsection~(\ref{widthFES}),
$\width_1(\cl)=\width_2(\cl)=\width_3(\cl)=n+2$.
However, compare with Proposition~\ref{prop-bounding}.
\end{example}

\begin{example}
\label{ex-noLex}
Finally, we would like to point out another aspect where the 3-dimensional situation
is worse compared to the one from Subsection~\ref{squareCase}.
Putting the right-hand Example from Figure~\ref{fig:planeSeq}, i.e.\
$\cl=[(0,0),\,(7,1),\,(8,1),\,(1,2)]$, and its transpose
$\cl'=[(0,0),\,(1,7),\,(1,8),\,(2,1)]$ into adjacent layers,
the resulting 3-dimensional exceptional sequence does not allow any
lexicographic, exceptional order.
\end{example}

\section{Mapping to the cube}\label{MapCube}

\subsection{The labeling}\label{labeling}
Assume that $\cl\subset\Z^r$ is a maximal exceptional sequence, i.e.,
in particular, it consists of $2^r$ points. Then, we consider the map
$$
\Phi:\cl \hookrightarrow \Z^r \surj (\Z/2\Z)^r.
$$

\begin{proposition}
\label{prop-cube}
The map \kbox{$\Phi$ is bijective}.
\end{proposition}

\begin{proof}
It suffices to check that $\Phi$ is injective. If
$\cl^i,\cl^j\in\cl$ with $i<j$ but $\Phi(\cl^i)=\Phi(\cl^j)$, then
$\cl^j-\cl^i\in 2\cdot\Z^r$, i.e.\
$(\cl^j-\cl^i)_\nu\in 2\Z$ for all coordinates $\nu=1,\ldots,r$.
This contradicts the condition stated in
Subsection~(\ref{adaptSpecial}) asking for
$(\cl^j-\cl^i)_\nu=1$.
\end{proof}

Since $\cl$ comes with a total ordering, it induces a labeling
of the vertices of the $r$-cube, i.e.\ of the elements of $(\Z/2\Z)^r$,
with the numbers $0,\ldots,2^r-1$.
For example, if the lower left vertex of the cube from
Figure~\ref{fig:cubeTetrahedra} is taken as the origin, then the
\fes\ leads to the following ordering of the elements of $(\Z/2\Z)^3$:
$$
\Phi= (0,1,1),\;(1,0,1),\;(1,1,0),\;(1,1,1),\;
      (0,0,0),\;(1,0,0),\;(0,1,0),\;(0,0,1).
$$

\begin{remark}
The map $\Phi$ stays bijective if we replace the \mes\ $\cl$ by a subset $\{\cl^0, \ldots, \cl^{2^r-1}\}$ of $\Z^r$,
having the less restrictive property that for each $i \not= j$ there is an index $\nu \in \{1, \ldots, r\}$ satisfying $|\cl^j_\nu - \cl^i_\nu| = 1$.\\
The subset of $\Z^2$ in Figure~\ref{fig:notOrderedSet} does not infect any other points of $\Z^2$, in particular it is not full.
This shows that the requirement that $\cl$ is {\em ordered} is essential.
\end{remark}

\begin{figure}[ht]
$
\begin{tikzpicture}[scale=0.4]
\draw[color=oliwkowy!40] (-0.3,-0.3) grid (3.3,3.3);
\fill[thick, color=black]
 (0,2) circle (5pt)
 (1,0) circle (5pt)
 (2,3) circle (5pt)
 (3,1) circle (5pt);
\end{tikzpicture}
$
\vspace*{-1ex}
\caption{A maximal subset (not a sequence) of $\Z^2$ that is not full}
\label{fig:notOrderedSet}
\end{figure}
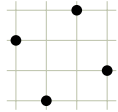

\begin{remark}
Let $\cl$ be an $r$-dimensional \mes.
Then the sequence $\tilde{\cl}$ defined by 
$\tilde{\cl}^i_k := -\cl^{2^r-1-i}_k$ is a \mes, too.
Moreover, $\tilde{\cl}$ is full if and only if $\cl$ is full.
The labeling of the vertices of the $r$-cube induced by $\tilde{\cl}$ is the reverse of the one induced by $\cl$.
\end{remark}
\subsection{Symmetries}\label{symmetries}
We have symmetries governed by the group
$(\Z/2\Z)^r$: A shift of $\cl\subseteq\Z^r$ by the unit vector
$e_\nu\in\Z^r$ for some $\nu=1,\ldots,r$
does neither change the status of exceptionality, maximality, or fullness,
i.e.\ infectivity. Within the cube $(\Z/2\Z)^r$, this operation is still given
by the map $\theta_\nu:x\mapsto x+e_\nu$, but it simply means
the reflection along the $\nu$-th hyperplane.
\\[1ex]
We may get rid of these symmetries by normalizing the labeling via
putting the first element $\cl^0$ of $\cl$ into
the lower left corner, i.e.\ asking for $\Phi(\cl^0)=0\in(\Z/2\Z)^r$
or even $\cl^0=0$.
For instance,
the two plane exceptional sequences from Figure~\ref{fig:planeSeq} lead
\begin{figure}[ht]
$
\begin{tikzpicture}[scale=0.5]
\fill[thick, color=black]
 (0,0) -- (2,0) (0,1) -- (2,1) (0,2) -- (2,2)
 (0,0) -- (0,2) (1,0) -- (1,2) (2,0) -- (2,2);
\draw[thin, color=black]
 (0.5,0.5) node{$0$} (1.5,0.5) node{$1$}
 (0.5,1.5) node{$3$} (1.5,1.5) node{$2$};
\end{tikzpicture}
\raisebox{2ex}{\mbox{ meaning }}
\begin{tikzpicture}[scale=0.5]
\fill[thick, color=black]
 (0,0) -- (2,0) (0,1) -- (2,1) (0,2) -- (2,2)
 (0,0) -- (0,2) (1,0) -- (1,2) (2,0) -- (2,2);
\draw[thin, color=black]
 (0.5,0.5) node{$\cl^0$} (1.5,0.5) node{$\cl^1$}
 (0.5,1.5) node{$\cl^3$} (1.5,1.5) node{$\cl^2$};
\end{tikzpicture}
\hspace{2em}
\raisebox{2ex}{\mbox{ and }}
\hspace{2em}
\begin{tikzpicture}[scale=0.5]
\fill[thick, color=black]
 (0,0) -- (2,0) (0,1) -- (2,1) (0,2) -- (2,2)
  (0,0) -- (0,2) (1,0) -- (1,2) (2,0) -- (2,2);
\draw[thin, color=black]
 (0.5,0.5) node{$0$} (1.5,0.5) node{$3$}
 (0.5,1.5) node{$2$} (1.5,1.5) node{$1$};
\end{tikzpicture}
\raisebox{2ex}{\mbox{ meaning }}
\begin{tikzpicture}[scale=0.5]
\fill[thick, color=black]
 (0,0) -- (2,0) (0,1) -- (2,1) (0,2) -- (2,2)
 (0,0) -- (0,2) (1,0) -- (1,2) (2,0) -- (2,2);
\draw[thin, color=black]
 (0.5,0.5) node{$\cl^0$} (1.5,0.5) node{$\cl^3$}
 (0.5,1.5) node{$\cl^2$} (1.5,1.5) node{$\cl^1$};
\end{tikzpicture}
$
\caption{Labelings for the \fes s of Figure~\ref{fig:planeSeq}}
\label{fig:labelPlaneSeq}
\end{figure}
to the normalized labelings presented in Figure~\ref{fig:labelPlaneSeq}.
Doing the same for the \fes\ from Figure~\ref{fig:cubeTetrahedra}
yields
the labeled cube encoded as two squares (lower/upper layer) in
Figure~\ref{fig:labelCubeTetrahedra}.
\begin{figure}[ht]
$
\begin{tikzpicture}[scale=0.5]
\fill[thick, color=black]
 (0,0) -- (2,0) (0,1) -- (2,1) (0,2) -- (2,2)
 (0,0) -- (0,2) (1,0) -- (1,2) (2,0) -- (2,2);
\draw[thin, color=black]
 (0.5,0.5) node{$4$} (1.5,0.5) node{$5$}
 (0.5,1.5) node{$6$} (1.5,1.5) node{$2$};
\end{tikzpicture}
\hspace{0.5em}\raisebox{2ex}{\mbox{ and }}\hspace{0.5em}
\begin{tikzpicture}[scale=0.5]
\fill[thick, color=black]
 (0,0) -- (2,0) (0,1) -- (2,1) (0,2) -- (2,2)
 (0,0) -- (0,2) (1,0) -- (1,2) (2,0) -- (2,2);
\draw[thin, color=black]
 (0.5,0.5) node{$7$} (1.5,0.5) node{$1$}
 (0.5,1.5) node{$0$} (1.5,1.5) node{$3$};
\end{tikzpicture}
\hspace{2em}\raisebox{2ex}{\mbox{ becomes normalized to }}\hspace{2em}
\begin{tikzpicture}[scale=0.5]
\fill[thick, color=black]
 (0,0) -- (2,0) (0,1) -- (2,1) (0,2) -- (2,2)
 (0,0) -- (0,2) (1,0) -- (1,2) (2,0) -- (2,2);
\draw[thin, color=black]
 (0.5,0.5) node{$0$} (1.5,0.5) node{$3$}
 (0.5,1.5) node{$7$} (1.5,1.5) node{$1$};
\end{tikzpicture}
\hspace{0.5em}\raisebox{2ex}{\mbox{ and }}\hspace{0.5em}
\begin{tikzpicture}[scale=0.5]
\fill[thick, color=black]
 (0,0) -- (2,0) (0,1) -- (2,1) (0,2) -- (2,2)
 (0,0) -- (0,2) (1,0) -- (1,2) (2,0) -- (2,2);
\draw[thin, color=black]
 (0.5,0.5) node{$6$} (1.5,0.5) node{$2$}
 (0.5,1.5) node{$4$} (1.5,1.5) node{$5$};
\end{tikzpicture}
$
\caption{Labelings for the \fes\ of Figure~\ref{fig:cubeTetrahedra}}
\label{fig:labelCubeTetrahedra}
\end{figure}
\\[1ex]
Finally, we should remark that, via permuting the coordinates,
the symmetric group $S_r$ acts, too. This action is visible on both the
exceptional subsets of $\Z^r$ as well as on the labelings of $(\Z/2\Z)^r$.
Altogether we have $(2^r)!/(2^r\cdot r!)=(2^r-1)!/r!$
(that is $840$ for $r=3$) cosets of possible labelings of
$(\Z/2\Z)^r$. It might be interesting to know which or how many of them
can be induced from \fes s.

\section{The configuration of layers}\label{confLayers}

Let $\cl$ be an $r$-dimensional exceptional sequence.
In the present section we choose and fix some $\nu\in\{1,\ldots,r\}$
and denote by $\pi_\nu:\Z^r\surj\Z$ the associated coordinate.
Moreover, we denote by $q_\nu:\Z^r\surj\Z^{r-1}$ the projection
forgetting the chosen $\nu$-th component.
Accordingly, we obtain a decomposition of $\cl$ into layers
$L_c=\cl\cap \pi_\nu^{-1}(c)$.
When it simplifies the notation, we may always assume, w.l.o.g., that $\nu=r$
is the last coordinate.

\subsection{Reduction to $(r-1)$-dimensional sequences}\label{reductR}
We start with the key lemma for reducing the dimension from $r$ to $r-1$.
At this point, no maximality of $\cl$ is assumed.

\begin{lemma}
\label{lem-combine}
Let $C\subset\Z$ be such that no two elements of $C$ are
adjacent to each other, i.e.\ $1\notin C-C$.
Then, the map $q_\nu$ is injective when restricted to
$\bigcup_{c\in C}L_c$. Moreover, its image
\kbox{$q_\nu(\bigcup_{c\in C}L_c)\subset\Z^{r-1}$}
with the inherited total order is an
$(r-1)$-dimensional exceptional sequence.
\end{lemma}

\begin{proof}
Consider $\cl^i,\cl^j\in\bigcup_{c\in C}L_c$ for some
$i<j$. Since the $\nu$-th coordinate of these elements belong to $C$,
we know for sure that
$(\cl^j-\cl^i)_\nu=\cl^j_\nu-\cl^i_\nu \neq1$. Thus,
by Subsection~(\ref{adaptSpecial}), there
has to be another index $\mu\neq\nu$ such that
$\cl^j_\mu-\cl^i_\mu =1$,
i.e.\ that $q_\nu(\cl^j)_\mu- q_\nu(\cl^i)_\mu=1$ (if $\mu < \nu$) or
$q_\nu(\cl^j)_{\mu-1}- q_\nu(\cl^i)_{\mu-1}=1$ ($\mu > \nu$).
This implies both that
$q_\nu(\cl^i)\neq q_\nu(\cl^j)$ and that
these two elements satisfy the exceptionality condition in $\Z^{r-1}$.
\end{proof}

\begin{corollary}
\label{cor-oddEven}
Assume that $A,B\subset\Z$ are disjoint subsets fulfilling the assumptions of
$C$ from Lemma~\ref{lem-combine}. Then, if $\cl\subset\Z^r$ is a \mes\ being
contained in $\bigcup_{c\in A\cup B}L_c$
{\rm (for instance, if $A\cup B=\Z$)}, both images
$\ko{\cl}_A:=q_\nu(\bigcup_{c\in A}L_c)$ and
$\ko{\cl}_B:=q_\nu(\bigcup_{c\in B}L_c)$ are \mes s in $\Z^{r-1}$.
\end{corollary}

\begin{proof}
On the one hand, it follows from Lemma~\ref{lem-combine} that
$\ko{\cl}_A$ and $\ko{\cl}_B$ are exceptional in $\Z^{r-1}$,
implying that $|\ko{\cl}_A|,|\ko{\cl}_B|\leq 2^{r-1}$.
On the other, we do also know that
$$
|\ko{\cl}_A|+|\ko{\cl}_B|=|\cl|=2^r.
$$
Hence, $|\ko{\cl}_A|=|\ko{\cl}_B|=2^{r-1}$, i.e.\
both $\ko{\cl}_A$ and $\ko{\cl}_B$ are
maximal exceptional.
\end{proof}

The main example for Corollary~\ref{cor-oddEven} is $A=2\Z$ and $B=2\Z+1$.
Thus, we obtain two $(r-1)$-dimensional \mes s
$\ko{\cl}_{\mbox{\tiny even}}$ and $\ko{\cl}_{\mbox{\tiny odd}}$
arising as the projections from the even and odd layers of $\cl$,
respectively.
Via the labeling map $\Phi$ established in
Subsection~(\ref{labeling}), the sequences $\ko{\cl}_{\mbox{\tiny even}}$
and $\ko{\cl}_{\mbox{\tiny odd}}$ correspond to facets, i.e.\ to
$(r-1)$-dimensional slices of the cube $(\Z/2\Z)^r$.

Note that a kind of an opposite implication is also true:
If $\cl_A$ and $\cl_B$ are two $(r-1)$-dimensional maximal or even full
exceptional sequences, then we obtain by
$$
\cl\;:=\;\big(\cl_A\times\{0\}\big) \;\disjcup\; \big(\cl_B\times\{1\}\big)
$$
an $r$-dimensional one of the same quality. Moreover, proceeding
with $\cl$, we recover
$\cl_A=\ko{\cl}_{\mbox{\tiny even}}$ and
$\cl_B=\ko{\cl}_{\mbox{\tiny odd}}$ we have started with.

\subsection{Thin sequences and maximal layers}\label{maxLayers}
Since we have still fixed a coordinate $\nu\in\{1,\ldots,r\}$,
we will call $\width(\cl):=\width_\nu(\cl)$ from
Subsection~(\ref{widthFES}) simply the width of $\cl$.
Then, for example, the \mes\ $\cl$ just built
from the lower-dimensional $\cl_A, \cl_B$
at the end of Subsection~(\ref{reductR})
has width $2$. And both layers are maximal, meaning that they
contain the maximal number of $2^{r-1}$ points.

\begin{proposition}
\label{prop-thinthree}
Let $\cl$ be a \mes\ in $\Z^r$.
Then,
its width is bounded by \kbox{$\width(\cl)\leq 3$}
if and only if $\cl$ contains a \kbox{maximal layer}.
\\[0.0ex]
Moreover, if this is the case, then $\cl$ is full, i.e.~contaminating the whole lattice $\Z^r$
-- provided that we already know that maximality
\kbox{implies fullness} in dimension $r-1$.
\end{proposition}

\begin{proof}
Assume that $\width(\cl)\leq 3$ and
denote by $A_0,B_1,A_2$ the three layers of $\cl$.
Then with $A=2\Z$ and $B=2\Z+1$ as in Subsection~(\ref{reductR}),
the layer $B_1$ equals the $(r-1)$-dimensional \mes\ $\ko{\cl}_B$ from
Corollary~\ref{cor-oddEven}.
\\[0.5ex]
Conversely, if $L_c$ is a full layer, then
any $L_d\neq\emptyset$ with $|d-c|\geq 2$ will yield a contradiction
via Lemma~\ref{lem-combine}: Just take $C:=\{c,d\}$.
\\[1.0ex]
Now, for the second part, let us assume that we have a maximal layer
$B_1$ sitting in between $A_0$ and $A_2$.
In particular, by the induction hypothesis,
all points of the associated hyperplane $[x_\nu=1]\subset\Z^r$ containing
$B_1$ will be contaminated.
\\[0.5ex]
In particular, this means that for each
$(a,0)\in A_0$ or $(a,2)\in A_2$, the points $(a,1)$ are contaminated,
In the second round, these pairs will infect the whole vertical lines
$a\times\Z$.
This, however, means that $\ko{\cl}_A\times\Z$ becomes contaminated.
\\[0.5ex]
Finally, as a third round,
we proceed with the contamination procedure for
$\ko{\cl}_A\subseteq\Z^{r-1}$ in each layer
$\Z^{r-1}\times\{c\}$ with $c\in\Z$ simultaneously.
\end{proof}

Since we have seen in Subsection~(\ref{squareCase}) that 2-dimensional
\mes s are always full, the previous proposition implies the same
fact for 3-dimensional \mes s of width (in {\em some} direction) at most~3.

\subsection{Heavy layers stick together}\label{largeStick}
Another immediate consequence of Lemma~\ref{lem-combine} is the
observation that heavy layers $L_c$, i.e.\ those having a large size
$\ell_c:=|L_c|$ will always be close to each other:

\begin{proposition}
\label{prop-largeStick}
Let $\cl\subset \Z^r$ be an $r$-dimensional exceptional sequence.
Then, any two different layers
\kbox{$L_a$ and $L_b$ with $\ell_a+\ell_b> 2^{r-1}$
have to be neighbors}, i.e.\ they satisfy $|b-a|=1$.
\end{proposition}

\begin{proof}
Assume that $L_a$ and $L_b$ are not neighbors. Then,
by Lemma~\ref{lem-combine} it follows that
$q_\nu:(L_a\cup L_b)\hookrightarrow\Z^{r-1}$ is injective, and that its image
is an exceptional sequence of dimension $r-1$.
Hence, $\ell_a+\ell_b\leq 2^{r-1}$.
\end{proof}

\subsection{Lower bounds for the heaviest layers}\label{heavyExist}
Especially for $r=3$, the previous Proposition~\ref{prop-largeStick}
has quite strong implications. The reason is the following statement ensuring
the existence of a layer $L$ with $|L| \geq 3$ in at least {\em one direction}
$\nu\in\{1,\ldots,r\}$:

\begin{lemma}\label{layersWithThree}
Let $\cl\subset\Z^r$ be an $r$-dimensional \mes.
Then there is a $\nu\in\{1,\ldots,r\}$ such that at least one of
the associated layers $L$ satisfies $\ell=|L|\geq\frac{2^r-1}{r}$.
In particular, if $r=3$, then there
\kbox{has to be a layer with $\ell\geq 3$}.
\end{lemma}

\begin{proof}
In the ordered sequence $\cl$ all elements after the initial one,
i.e.~after $\cl^0$,
have to be in one of the following layers
$$
(\{\cl_1^0+1\}\times\Z\times\ldots\times\Z),
(\Z\times\{\cl_2^0+1\}\times\Z\times\ldots\times\Z),
\ldots,
(\Z\times\ldots\times\Z\times\{\cl_r^0+1\}).
$$
Since there are $2^r-1$ elements from $\cl\setminus\{\cl^0\}$
to be distributed to $r$ possible layers, the
result follows from the pigeon hole principle.
\end{proof}

Finally, let us focus on the case $r=3$.
While we already have treated the case of full layers, i.e.\
$\ell_c=4$, in Proposition~\ref{prop-thinthree},
we may assume that $\ell_c\leq 3$ for all
$c\in\Z$ and all directions $\nu\in\{1,2,3\}$.
But then, Lemma~\ref{layersWithThree} guarantees the existence of
at least one layer $L_c$ (in some direction $\nu$) satisfying $\ell_c=3$.
Moreover, by Proposition~\ref{prop-largeStick}, any further layer
(in the same direction $\nu$) containing more
than one point has to be adjacent to $L_c$.

\section{Block structure} \label{blockStruct}

As in Section~\ref{confLayers}, we start with an $r$-dimensional exceptional
sequence $\cl$ and fix a coordinate $\nu\in\{1,\ldots,r\}$.
Recall that $\pi_\nu:\Z^r\surj\Z$ and $q_\nu:\Z^r\surj\Z^{r-1}$
denote the associated projections.
According to this, the non-empty layers
$L_c=\cl\cap \pi_\nu^{-1}(c)$ of $\cl$, i.e.\ those with $\ell_c=|L_c|\geq 1$,
can be arranged in several blocks. The following definition makes this precise:

\begin{definition}
A finite subset $B\subset\Z$ of consecutive numbers is called a segment for
$\nu$ if $L_b\neq\emptyset$ for all $b\in B$ and if
$L_{\min(B)-1}=L_{\max(B)+1}=\emptyset$. The preimage
$$
\textstyle
L_B:=\cl\cap \pi_\nu^{-1}(B)=\bigcup_{b\in B} L_b
$$
is called the block associated to $B$.
The layers $L_{\min(B)}$ and $L_{\max(B)}$ are called the boundary
or outer layers of $B$.
\end{definition}

\subsection{Outer layers are lighter than their neighbors}
\label{outerInner}
The following proposition shows that the number of elements in a layer
decreases towards the boundaries of the corresponding block.

\begin{proposition}
\label{prop-outerInner}
Let $\cl$ be an $r$-dimensional \mes. Then, for any direction
$\nu\in\{1,\ldots,r\}$ and for each segment $B\subset\Z$,
\kbox{no outer layer} of $L_B$ contains \kbox{more points}
than its \kbox{neighboring ``inner''} layer. That is,
$$
\ell_{\min(B)}\leq \ell_{\min(B)+1}
\hspace{1em}\mbox{and}\hspace{1em}
\ell_{\max(B)}\leq \ell_{\max(B)-1}.
$$
In particular, the width of the block $\width(L_B)=\max(B)-\min(B)+1$
does always exceed one.
\end{proposition}

Note that blocks may have width $\width(L_B)=2$. Then, the neighbors
$L_{\min(B)+1}$ and $L_{\max(B)-1}$ are obviously not inner layers at all
-- nevertheless, they were called so in the previous proposition.
In this special case, it implies that
$\ell_{\min(B)}=\ell_{\max(B)}$.
Anyway, for all cases, this proposition is a direct consequence
(just set $b:=\min(B)$ or $b:=\max(B)$) of the following lemma:

\begin{lemma}
\label{lem-arithMean}
Let $\cl$ be an $r$-dimensional \mes. Then, for any direction
$\nu\in\{1,\ldots,r\}$ and any $b\in\Z$ we have
$\ell_b\leq \ell_{b-1}+\ell_{b+1}$.
\end{lemma}

\begin{proof}
We start with the subsets
$$
C:=(b+1)+2\Z
\hspace{1em}\mbox{and}\hspace{1em}
C':= b +2\Z= C+1.
$$
They fit the assumptions of Corollary~\ref{cor-oddEven},
hence $\cl_C:=\cl\cap \pi_\nu^{-1}(C)=\bigcup_{c\in C} L_c$
consists of exactly $2^{r-1}$ elements -- and so does $\cl_{C'}$ what,
nevertheless, we do not need. Instead, we consider
$$
C'':= \{b\} \cup C \setminus\{b\pm 1\},
$$
i.e.~
$
\cl_{C''} = L_{b} \cup \cl_C \setminus (L_{b-1}\cup L_{b+1}).
$
Furthermore, $C''\subset\Z$ meets the requirements of Lemma~\ref{lem-combine},
thus $|\cl_{C''}|\leq 2^{r-1}$.
Hence
$
2^{r-1} \geq |\cl_{C''}|
= |L_{b} \cup \cl_C \setminus (L_{b-1}\cup L_{b+1})|
= \ell_b + |\cl_{C}| - \ell_{b-1} - \ell_{b+1} = \ell_b + 2^{r-1} - \ell_{b-1} - \ell_{b+1}
$
and therefore
$
\ell_b \leq \ell_{b-1} + \ell_{b+1}
$.
\end{proof}

\subsection{Blocks are balanced}\label{blocksBalanced}
In Corollary~\ref{cor-oddEven} we had seen that the even and the odd layers
contribute the same number of points to a \mes~$\cl$,
namely $2^{r-1}$ in each case. However,
a similar statement is true for each block separately.

\begin{proposition}
\label{prop-alternatingLayers}
Let $\cl$ be a \mes. Then,
within each block $L_B$, the alternating sums $\sum_{c\in B}(-1)^c\,\ell_c$
vanish.
\end{proposition}

Note that this proposition
excludes once again, like it does Proposition~\ref{prop-outerInner},
the existence of blocks of width one.

\begin{proof}
We do again exploit Corollary~\ref{cor-oddEven}.
We start with the standard pair $C:=2\Z$ and $C':=2\Z+1$.
We then produce $D,D'\subset\Z$ out of $C,C'\subset\Z$ by
switching $B\cap C$ with $B\cap C'$ inside these sets.
More precisely, we set
$$
D:= C \cup (B\cap C') \setminus (B\cap C) \setminus\{\min(B)-1,\,\max(B)+1\}
$$
and
$$
D':= C' \cup (B\cap C) \setminus (B\cap C') \setminus\{\min(B)-1,\,\max(B)+1\}.
$$
These sets are still disjoint, and the removal of the empty layers beyond $B$
ensures the assumptions $1\notin D-D$ and $1\notin D'-D'$ without
destroying the property $\cl\subseteq \bigcup_{c\in D\cup D'} L_c$.
Thus, both $\cl_C$ and $\cl_D$ (and also $\cl_{C'}$ and $\cl_{D'}$)
contain $2^{r-1}$ elements each.
But this implies that
$\sum_{c\in B\cap C}\ell_c=\sum_{c\in B\cap C'}\ell_c$.
\end{proof}

\subsection{Empty layers}\label{emptyLayers}
The presence of inner empty layers, i.e.\ those separating blocks in $\cl$,
allows to manipulate $\cl$ in different ways
with various results. To demonstrate this, we assume
that \kbox{$L_0=\emptyset$}.
As usual, this is meant for a fixed direction $\nu\in\{1,\ldots,r\}$.
Recall from Definition~\ref{def-contamination} that $e_\nu\in\Z^r$
denotes the corresponding unit vector, i.e.,
$\pi_\nu(e_\nu)=1$ and $q_\nu(e_\nu)=0\in\Z^{r-1}$.

\subsubsection{Merging positive and negative layers}\label{mergeLayers}
Fix a natural number $m\in\N$. Then, we obtain a new sequence
$\cl(m)$ out of $\cl$ by defining
$$
\cl(m)^i:= \left\{\begin{array}{ll}\cl^i & \mbox{if }\cl_\nu^i > 0
\\[0.3ex]
\cl^i+m\cdot e_\nu & \mbox{if }\cl_\nu^i < 0.
\end{array}\right.
$$
While $\cl(0)=\cl$, the new sequence $\cl(1)$ arises from $\cl$ by
simply erasing the
$0$-th (empty) layer. For larger $m\geq 2$, it will more and more happen that
the former negative part $\cl^-$ and the former positive part $\cl^+$
merge. In particular, a potential block structure of $\cl$ will disappear --
or reappear, in a different manner, for $m\gg 0$.
\\[1ex]
While it is clear that $\cl(m)$ stays a \mes\ if $\cl$ was one, it should be
mentioned that it is {\em not} guaranteed at all (and almost always false)
that exceptionality survives when including an empty layer in an
exceptional sequence.
Similarly, a potential fullness of $\cl$ is bequeathed to $\cl(m)$ -- but
the reverse conclusion does not work either.

\subsubsection{Duplicating or reducing empty layers}
\label{duplicateEmptyLayers}
A much easier situation arises when the definition of $\cl(m)$ is
literally extended to negative $m<0$. This means to amend the empty layer
$L_0$ by additional $|m|$ empty ones. The reverse operation means to
thin out sequences of consecutive empty layers such that at least
one of them survives. This operation is a special type of those
from Subsection~(\ref{mergeLayers}).

\begin{proposition}
\label{prop-DupRedLayers}
Duplicating or reducing empty layers
in a sequence $\cl$ (without extinguishing them all) leads to a sequence
$\cl'$ sharing with $\cl$ exactly the same properties with
regard to exceptionality, maximality, and infectivity, i.e.\ fullness.
\end{proposition}

\begin{proof}
The property of $(\cl^j-\cl^i)_\mu$ being equal or not equal to $1$
is equivalent to the same property of $(\cl'^{j}-\cl'^{i})_\mu$
for every $\mu=1,\ldots,r$. Moreover, within the contamination procedure,
all new points $(p,c,q)\in\Z^{\nu-1}\times\Z\times\Z^{r-\nu}$ arising in the empty layers
$L_c$ appear simultaneously in all layers, i.e.\ as a line $p\times\Z\times q$.
\end{proof}

The previous observation means that, while empty layers do
indeed matter, we can
restrict ourselves to the case of isolated (but, maybe, still several) ones.
This applies either for the classification of \mes s or
for proving that maximality implies fullness.

\subsection{Bounding the sequence}\label{packSeq}
The previous results have the consequence that the problem of fullness of all \mes\ can
be
reduced to \mes\ in a bounded region.

\begin{proposition}
\label{prop-bounding}
All $r$-dimensional \mes\ are full if and only all those contained in a cube with
$3 \cdot 2^{r-1} - 1$ layers in each direction are full.
\end{proposition}

\begin{proof}
By Proposition~\ref{prop-outerInner}, all blocks of a \mes\ $\cl$ have at least width~2
in any direction.
By Proposition~\ref{prop-DupRedLayers}, we can exchange $\cl$ by $\cl'$ that has no empty
blocks of width larger than 1 and still the same properties with regard to fullness.
Hence $\cl'$ has at most $2^r$ non-empty layers and at most $2^{r-1} - 1$ empty layers in
between in all $\nu$-directions.
Thus $\cl'$ fits in a cube with $2^r+2^{r-1}-1$ layers in each direction.
\end{proof}

\section{Special issues for $r=3$} \label{compRThree}

In this concluding section, we assume that $\cl$ is a
3-dimensional \mes, i.e.~with $r=3$. Moreover,
by Proposition~\ref{prop-thinthree},
we may and will assume that $\cl$ has no layer of size $4$.
Finally, according to
Subsection~(\ref{duplicateEmptyLayers}),
we suppose that all inner empty layers, in all directions, are isolated.
\\[1ex]
We are going to present upper bounds
for the widths of $\cl$ with respect to all three directions. Let $\nu\in\{1,2,3\}$
-- we look for all possible ``load sequences''
$(\ldots,\ell_{c-1},\ell_c,\ell_{c+1},\ldots)$
indicating the sizes of the consecutive layers, up to equivalence relations
like shifts or reflections within each block.

\subsection{Assuming a layer of size~3}
\label{estimateWidthThree}
For the fixed direction $\nu$ we suppose that we have a layer of size~3
within the \mes~$\cl$. We are going to distinguish a few cases
-- but the overall result will be that $\width_\nu(\cl)\leq 6$.
Moreover, if lacking empty layers, we can bound it by $5$.

\subsubsection{Two layers of size 3}\label{twoLayersThree}
If $\cl$ contains, in the same direction $\nu$, two layers of size~3, then, by
Proposition~\ref{prop-largeStick}, they have to be adjacent.
The remaining two layers of width 1 cannot be isolated by
Proposition~\ref{prop-outerInner}. Thus, we obtain as possible
load sequences $(1,3,3,1)$ or $(3,3,1,1)$, or $(3,3,0,1,1)$.
In any case, we have $\width_\nu(\cl)\leq 5$.

\subsubsection{Two layers of size 2}
\label{threeTwoLayersTwo}
They are supposed to exist additionally to the layer of size 3 we have anyway.
Again by Proposition~\ref{prop-largeStick}, we see that the
load sequence has to look like $(\ldots,2,3,2,\ldots)$, thus, it has to be
$(2,3,2,1)$ yielding $\width_\nu(\cl)=4$.

\subsubsection{Exactly one layer of size 2}
\label{twoThreeLayer}
Here we obtain $(1,3,2)$ as a necessary subsequence. The remaining part
has to be $(1,1)$ -- either separated from the big part by an empty layer or
not. We obtain $\width_\nu(\cl)=5$ or $6$.

\subsubsection{All layers are of size 1}
\label{threeRestOne}
This is meant up to the big layer of size 3 we have assumed anyway.
Then, this case cannot occur --
it contradicts Lemma~\ref{lem-arithMean}.

\subsection{No layer of size~3}
\label{estimateWidthNotThree}
By Proposition~\ref{prop-bounding}, the longest possible load sequence is
$$
(1,1,0,1,1,0,1,1,0,1,1);
$$
it has length~$11$. And this can indeed occur in a \mes.
On the other hand, if there is at least one layer of size~2
involved, then one of the subsequences $(1,1,0,1,1)$ has to be replaced
by $(1,2,1)$ or even $(2,2)$. Thus, the estimate for $\width_\nu(\cl)$
drops from $11$ to $9$.

\subsection{Computational evidence}
\label{compEv}
We will conclude our proof of Theorem~\ref{th-bbbCont} (or, equivalently,
Theorem~\ref{th-bbb}) by strengthening the claim of Lemma~\ref{layersWithThree}.
The starting point was that
the seven elements of $\{\cl^1,\cl^2,\ldots,\cl^7\}$ are distributed
on three affine planes having distance one to $\cl^0$. While the
original claim of the lemma states
that there is one plane containing three points of
$\cl\setminus\{\cl^0\}$, we do even know that the distribution must be either
$(3,3,3)$, $(3,3,2)$, $(3,3,1)$ or $(3,2,2)$ -- recall that there are no planes containing 4 points
of $\cl$.
\\[1ex]
Assume first that the distribution is $(3,3,3)$, $(3,3,2)$ or $(3,3,1)$. This means that, w.l.o.g.\
for $\nu=1,2$, we may apply the scenario of
Subsection~(\ref{estimateWidthThree}) leading to $\width_{1/2}(\cl)\leq 6$
and, for $\nu=3$, the first one from Subsection~(\ref{estimateWidthNotThree})
leading to $\width_3(\cl)\leq 11$.
\\[1ex]
Similarly, the distribution $(3,2,2)$ leads to, w.l.o.g.,
$\width_{1}(\cl)\leq 6$ and $\width_{2/3}(\cl)\leq 9$.
\\[1ex]
Altogether this means that $\cl$ fits either in a
$(6\times 6\times 11)$- or in a
$(6\times 9\times 9)$-grid.
Both cases had been checked with a computer, using a rather straight
algorithm. The result was that all possible \mes s had contaminated the whole
space $\Z^3$ in at most seven steps.

\bibliographystyle{alpha}
\bibliography{fes3d}

\begin{thebibliography}{ABKW20}

\bibitem[ABKW20]{immaculate}
Klaus {Altmann}, Jaros{\l}aw {Buczy\'nski}, Lars {Kastner}, and Anna-Lena
  {Winz}.
\newblock Immaculate line bundles on toric varieties.
\newblock {\em Pure and Applied Mathematics Quarterly}, 16(4):1147–--1217,
  2020.

\bibitem[AW]{fesFano}
Klaus Altmann and Frederik Witt.
\newblock {The structure of exceptional sequences on toric varieties of Picard
  rank two}.
\newblock {arXiv:2112.14637 [math.AG]}.

\bibitem[BGKS15]{Phantom}
Christian {B\"ohning}, Hans-Christian {Graf von Bothmer}, Ludmil {Katzarkov},
  and Pawel {Sosna}.
\newblock {Determinantal Barlow surfaces and phantom categories.}
\newblock {\em {J. Eur. Math. Soc. (JEMS)}}, 17(7):1569--1592, 2015.

\bibitem[{Efi}14]{efimov}
Alexander~I. {Efimov}.
\newblock {Maximal lengths of exceptional collections of line bundles.}
\newblock {\em {J. Lond. Math. Soc., II. Ser.}}, 90(2):350--372, 2014.

\bibitem[{Kaw}06]{kaw1}
Yuriro {Kawamata}.
\newblock {Derived categories of toric varieties.}
\newblock {\em {Mich. Math. J.}}, 54(3):517--535, 2006.

\bibitem[{Kaw}13]{kaw2}
Yujiro {Kawamata}.
\newblock {Derived categories of toric varieties. II.}
\newblock {\em {Mich. Math. J.}}, 62(2):353--363, 2013.

\bibitem[{Kaw}16]{kaw3}
Yujiro {Kawamata}.
\newblock {Derived categories of toric varieties. III.}
\newblock {\em {Eur. J. Math.}}, 2(1):196--207, 2016.

\bibitem[Lee]{pic2Fano}
Dae-Won Lee.
\newblock {Classification of full exceptional collections on smooth toric Fano
  varieties with Picard rank two}.
\newblock {arXiv:2005.09783v2 [math.AG]}.

\bibitem[{Mir}21]{Mironov}
Mikhail {Mironov}.
\newblock {Lefschetz exceptional collections in \(S_k\)-equivariant categories
  of \((\mathbb{P}^n)^k\)}.
\newblock {\em {Eur. J. Math.}}, 7(3):1182--1208, 2021.

\bibitem[{Sos}20]{summary_Phantoms}
Pawel {Sosna}.
\newblock {Some remarks on phantom categories and motives}.
\newblock {\em {Bull. Belg. Math. Soc. - Simon Stevin}}, 27(3):337--352, 2020.

\end{thebibliography}

\end{document}